\documentclass[a4paper,reqno,10pt]{amsart}

\usepackage{hyperref}
\usepackage{a4wide}

\usepackage[foot]{amsaddr}

\usepackage{amssymb}
\usepackage{amstext}
\usepackage{amsmath}
\usepackage{amscd}
\usepackage{amsthm}
\usepackage{amsfonts}

\usepackage{graphicx}
\usepackage{latexsym}
\usepackage{mathrsfs}

\usepackage{enumitem}
\setlist[enumerate,1]{label={\upshape(\arabic*)}}
\setlist[enumerate,2]{label={\upshape(\alph*)}}

\usepackage{tikz}
\usetikzlibrary{cd,arrows,matrix,backgrounds,shapes,positioning,calc,decorations.markings,decorations.pathmorphing,decorations.pathreplacing}
\tikzset{blackv/.style={circle,fill=black,inner sep=3pt,outer sep=3pt},
         whitev/.style={circle,fill=white,draw=black,inner sep=3pt,outer sep=3pt},
         blabel/.style={circle,draw=black,inner sep=1.5pt,outer sep=0pt},
         redv/.style={circle,fill=red,inner sep=3pt,outer sep=3pt},
         block/.style={draw,rectangle split,rectangle split horizontal,rectangle split parts=#1},
         symbol/.style={
           draw=none,
           every to/.append style={
             edge node={node [sloped, allow upside down, auto=false]{$#1$}}}}
}

\usepackage{array}
\newcolumntype{C}{>{$}c<{$}}

\newtheorem{theorem}{Theorem}[section]
\newtheorem{theoremi}{Theorem}

\newtheorem{corollary}[theorem]{Corollary}
\newtheorem{lemma}[theorem]{Lemma}
\newtheorem*{lemma*}{Lemma}
\newtheorem*{theorem*}{Theorem}
\newtheorem{proposition}[theorem]{Proposition}
\newtheorem{definition-proposition}[theorem]{Definition-Proposition}

\newtheorem{question}[theorem]{Question}

\theoremstyle{definition}
\newtheorem{definition}[theorem]{Definition}

\newtheorem{remark}[theorem]{Remark}
\newtheorem{example}[theorem]{Example}

\newtheorem*{ack}{Acknowledgments}
\newtheorem*{conv}{Conventions and notation}
\newtheorem*{org}{Organization}

\newcommand{\qedb}{\hfill\blacksquare}

\newcommand{\la}{\langle}
\newcommand{\ra}{\rangle}
\newcommand{\al}{\alpha}
\newcommand{\be}{\beta}
\newcommand{\ga}{\gamma}

\renewcommand{\AA}{\mathcal{A}}
\newcommand{\CC}{\mathcal{C}}
\newcommand{\GG}{\mathcal{G}}
\newcommand{\DD}{\mathcal{D}}
\newcommand{\FF}{\mathcal{F}}
\newcommand{\FFF}{\mathsf{F}}
\newcommand{\HH}{\mathcal{H}}

\newcommand{\PP}{\mathcal{P}}

\renewcommand{\SS}{\mathcal{S}}
\newcommand{\TT}{\mathcal{T}}
\newcommand{\TTT}{\mathsf{T}}
\newcommand{\UU}{\mathcal{U}}
\newcommand{\WW}{\mathcal{W}}

\renewcommand{\P}{\mathbb{P}}

\newcommand{\Ext}{\operatorname{Ext}\nolimits}
\newcommand{\Hom}{\operatorname{Hom}\nolimits}
\newcommand{\End}{\operatorname{End}\nolimits}
\newcommand{\op}{\operatorname{op}\nolimits}

\newcommand{\Image}{\operatorname{Im}\nolimits}
\newcommand{\Kernel}{\operatorname{Ker}\nolimits}
\newcommand{\Cokernel}{\operatorname{Coker}\nolimits}

\newcommand{\coker}{\Cokernel}
\newcommand{\im}{\Image}
\renewcommand{\ker}{\Kernel}
\newcommand{\un}{\underline}
\newcommand{\ov}{\overline}

\newcommand{\ot}{\leftarrow}

\DeclareMathOperator{\tr}{\mathsf{tr}}

\DeclareMathOperator{\moduleCategory}{\mathsf{mod}} \renewcommand{\mod}{\moduleCategory}

\DeclareMathOperator{\proj}{\mathsf{proj}}

\DeclareMathOperator{\ind}{\mathsf{ind}}
\DeclareMathOperator{\simp}{\mathsf{sim}}

\DeclareMathOperator{\ice}{\mathsf{ice}}
\DeclareMathOperator{\icep}{\mathsf{ice_p}}
\DeclareMathOperator{\dfice}{\mathsf{df-ice}}

\DeclareMathOperator{\ftors}{\mathsf{f-tors}}
\DeclareMathOperator{\stors}{\mathsf{s-tors}}
\DeclareMathOperator{\sftors}{\mathsf{sf-tors}}
\DeclareMathOperator{\wide}{\mathsf{wide}}
\DeclareMathOperator{\fwide}{\mathsf{f-wide}}
\DeclareMathOperator{\tors}{\mathsf{tors}}
\DeclareMathOperator{\torf}{\mathsf{torf}}

\DeclareMathOperator{\ttilt}{\tau\mathsf{-tilt}}
\DeclareMathOperator{\sttilt}{\mathsf{s}\tau\mathsf{-tilt}}
\DeclareMathOperator{\ttiltp}{\tau\mathsf{-tilt-pair}}
\DeclareMathOperator{\wttilt}{\mathsf{w}\tau\mathsf{-tilt}}
\DeclareMathOperator{\wttiltp}{\mathsf{w}\tau\mathsf{-tilt-pair}}
\DeclareMathOperator{\rigid}{\mathsf{rigid}}

\DeclareMathOperator{\sint}{\mathsf{s-int}}
\DeclareMathOperator{\Hasse}{\mathsf{Hasse}}

\DeclareMathOperator{\Sub}{\mathsf{Sub}}
\DeclareMathOperator{\Fac}{\mathsf{Fac}}
\DeclareMathOperator{\Filt}{\mathsf{Filt}}
\DeclareMathOperator{\add}{\mathsf{add}}
\DeclareMathOperator{\WR}{\mathsf{W}_R}

\DeclareMathOperator{\ccok}{\mathsf{cok}}
\DeclareMathOperator{\kker}{\mathsf{ker}}

\DeclareMathOperator{\supp}{\mathsf{supp}}

\newcommand{\iso}{\cong}

\newcommand{\defl}{\twoheadrightarrow}

\newcommand{\equi}{\simeq}
\newcommand{\sst}[1]{\substack{#1}}

\numberwithin{equation}{section}

\begin{document}
\title[ICE-closed subcategories and wide $\tau$-tilting modules]{ICE-closed subcategories and wide $\tau$-tilting modules}

\author[H. Enomoto]{Haruhisa Enomoto}
\address{H. Enomoto:  Graduate School of Science, Osaka Prefecture University, 1-1 Gakuen-cho, Naka-ku, Sakai, Osaka 599-8531, Japan}
\email{{\rm H. Enomoto: }the35883@osakafu-u.ac.jp}

\author[A. Sakai]{Arashi Sakai}
\address{A. Sakai: Graduate School of Mathematics, Nagoya University, Chikusa-ku, Nagoya. 464-8602, Japan}
\email{{\rm A. Sakai: }m20019b@math.nagoya-u.ac.jp}
\begin{abstract}
  In this paper, we study ICE-closed (= Image-Cokernel-Extension-closed) subcategories of an abelian length category using torsion classes.
  To each interval $[\mathcal{U},\mathcal{T}]$ in the lattice of torsion classes, we associate a subcategory $\mathcal{T} \cap \mathcal{U}^\perp$ called the heart. We show that every ICE-closed subcategory can be realized as a heart of some interval of torsion classes, and give a lattice-theoretic characterization of intervals whose hearts are ICE-closed.
  In particular, we prove that ICE-closed subcategories are precisely torsion classes in some wide subcategories.
  For an artin algebra, we introduce the notion of wide $\tau$-tilting modules as a generalization of support $\tau$-tilting modules. Then we establish a bijection between wide $\tau$-tilting modules and doubly functorially finite ICE-closed subcategories, which extends Adachi--Iyama--Reiten's bijection on torsion classes.
  For the hereditary case, we discuss the Hasse quiver of the poset of ICE-closed subcategories by introducing a mutation of rigid modules.
\end{abstract}

\maketitle

\tableofcontents

\section{Introduction}\label{sec:1}
\subsection{Background}
The notion of \emph{torsion classes} in an abelian category was introduced by Dickson \cite{dickson} to generalize the notion of torsion abelian groups. They are subcategories closed under taking extensions and quotients, and they have been playing an important role in the representation theory of algebras from various viewpoints, e.g. tilting theory \cite{brebut}, $t$-structures \cite{HRS}, cluster theory \cite{IT}, stability conditions \cite{brid} and so on.
Recently, \cite{AIR} established a bijection between functorially finite torsion classes in $\mod\Lambda$ for an artin algebra $\Lambda$ and so-called \emph{support $\tau$-tilting modules}. Since then, the study of torsion classes via $\tau$-tilting modules has been one of the main topics in this area, e.g. \cite{MS,asai,DIJ}.

Besides torsion classes, \emph{wide subcategories} of an abelian category, that is, exact abelian subcategories closed under extensions, have been studied in various contexts, e.g. stability conditions \cite{BKT}, ring epimorphisms and universal localizations \cite{MS}, bijections with torsion classes \cite{IT,MS,enomono} and classifications \cite{hovey,ringel}.

Recently, the notion of \emph{ICE-closed subcategories} of an abelian category was introduced and studied by the first author in \cite{enomono,rigid}, which are subcategories closed under taking Images, Cokernels and Extensions (Definition \ref{def:basicdef}).
Both torsion classes and wide subcategories are typical examples of ICE-closed subcategories. In \cite{rigid}, for a Dynkin quiver $Q$, a bijection between ICE-closed subcategories of $\mod kQ$ and rigid $kQ$-modules was established, which extends a bijection between torsion classes and support ($\tau$-)tilting modules given in \cite{IT,AIR}. This motivates us to study ICE-closed subcategories and find a similar bijection for the non-hereditary case.

\subsection{Main results}
In this paper, we study ICE-closed subcategories of an abelian length category using torsion classes and $\tau$-tilting theory.
It is easily checked that every torsion class in a wide subcategory (viewed as an abelian category) is ICE-closed, see Lemma \ref{lem:torsice}. One of the main results of this paper is that the converse holds:
\begin{theoremi}[= Corollary \ref{cor:widetors}] \label{thm:a}
  Let $\AA$ be an abelian length category and $\CC$ a subcategory of $\AA$. Then the following are equivalent.
  \begin{enumerate}
    \item $\CC$ is an ICE-closed subcategory of $\AA$.
    \item There is some wide subcategory $\WW$ of $\AA$ such that $\CC$ is a torsion class in $\WW$.
  \end{enumerate}
\end{theoremi}

For the proof of this and further investigation, the \emph{hearts of twin torsion pairs} introduced by Tattar \cite{tattar} are quite useful. For an interval $[\UU, \TT]$ in the lattice $\tors\AA$ of torsion classes in an abelian length category $\AA$, we define a subcategory $\HH_{[\UU,\TT]}:= \TT \cap \UU^\perp$.
This subcategory played a key role in several papers \cite{DIRRT,AP,tattar}, and we call it the \emph{heart} of $[\UU,\TT]$ following \cite{tattar}.
It is known that every wide subcategory can be realized as a heart of some  interval, and such an interval called a \emph{wide interval} was investigated by Asai and Pfeifer \cite{AP}. In this paper, we obtain the following result on ICE-closed subcategories in this direction.
\begin{theoremi}[= Proposition \ref{prop:arashi}, Theorem \ref{thm:iceintchar}] \label{thm:b}
  Let $\AA$ be an abelian length category.
  \begin{enumerate}
    \item Every ICE-closed subcategory can be realized as a heart of some interval in $\tors\AA$.
    \item Let $[\UU,\TT]$ be an interval in $\tors\AA$. Then the following are equivalent.
    \begin{enumerate}
      \item $[\UU,\TT]$ is an ICE interval, that is, its heart is an ICE-closed subcategory of $\AA$.
      \item There is some $\TT' \in \tors\AA$ with $\UU \subseteq \TT \subseteq \TT'$such that $[\UU,\TT']$ is a wide interval.
      \item We have $\TT \subseteq \UU^+$, where $\UU^+:= \UU \vee \bigvee \{ \UU' \in \tors\AA \, | \, \text{there is a Hasse arrow $\UU' \to \UU$} \}$.
    \end{enumerate}
  \end{enumerate}
\end{theoremi}
We deduce Theorem \ref{thm:a} by using this theorem. Moreover, this theorem enables us to construct all ICE-closed subcategories of $\AA$ as hearts when $\tors\AA$ is given.

Next, we focus on the case where $\AA = \mod\Lambda$ for an artin algebra $\Lambda$.
By Theorem \ref{thm:a}, every ICE-closed subcategory of $\mod\Lambda$ is a torsion class in some wide subcategory of $\mod\Lambda$. From the viewpoint of $\tau$-tilting theory, it is natural to focus on ICE-closed subcategories which are \emph{functorially finite} torsion classes in some \emph{functorially finite} wide subcategories, and we call such ICE-closed subcategories \emph{doubly functorially finite} (Definition \ref{def:dfice}).

On the other side, we call a $\Lambda$-module \emph{wide $\tau$-tilting} if it is $\tau_\WW$-tilting in some functorially finite wide subcategory $\WW$ (Definition \ref{def:wtt}). This notion generalizes support $\tau$-tilting modules, see Proposition \ref{prop:serrettilt}.
In this setting, we establish the following bijection.
\begin{theoremi}[= Theorem \ref{thm:wttiltbij}, Proposition \ref{prop:ttfice}]\label{thm:c}
  Let $\Lambda$ be an artin algebra. Then we have a bijection between the following sets.
  \begin{enumerate}
    \item The set $\dfice\Lambda$ of doubly functorially finite ICE-closed subcategories of $\mod\Lambda$.
    \item The set $\wttilt\Lambda$ of isomorphism classes of basic wide $\tau$-tilting $\Lambda$-modules.
  \end{enumerate}
  Moreover, if $\Lambda$ is $\tau$-tilting finite, then every ICE-closed subcategory is doubly functorially finite, and there are only finitely many ICE-closed subcategories.
\end{theoremi}
The bijection sends $\CC \in \dfice\Lambda$ to the $\Ext$-progenerator $P(\CC)$ of $\CC$, and $M \in \wttilt\Lambda$ to $\ccok M$, the subcategory of $\mod\Lambda$ consisting of cokernels of maps in $\add M$. This bijection extends Adachi--Iyama--Reiten's bijection, namely, we have the following commutative diagram whose horizontal maps are bijective (Corollary \ref{cor:compati}):
\[
\begin{tikzcd}
  \wttilt\Lambda \rar[shift left, "\ccok"] & \dfice\Lambda \lar[shift left, "P(-)"] \\
  \sttilt\Lambda \rar[shift left, "\Fac"] \uar[symbol=\subseteq] & \ftors\Lambda \uar[symbol=\subseteq] \lar[shift left, "P(-)"]
\end{tikzcd}
\]
Moreover, if $\Lambda$ is hereditary, then this recovers the result in \cite{rigid} mentioned above, since wide $\tau$-tilting modules are precisely rigid modules in this case (Proposition \ref{prop:heredwttilt}).

By combining Theorems \ref{thm:b} and \ref{thm:c}, we obtain an algorithm to compute all wide $\tau$-tilting modules if $\Lambda$ is $\tau$-tilting finite and the poset of support $\tau$-tilting modules is given, see Corollary \ref{cor:wttiltfromint}. It seems to be an interesting but difficult problem to give a homological characterization of wide $\tau$-tilting modules except for the hereditary case.

Finally, we investigate the Hasse quiver of the poset $\dfice\Lambda$ ordered by inclusion. By the isomorphism $\ccok \colon \wttilt\Lambda \xrightarrow{\sim} \dfice\Lambda$ in Theorem \ref{thm:c}, we can endow $\wttilt\Lambda$ with the poset structure, namely, $M \leq N$ if $\ccok M \subseteq \ccok N$ holds.
As for $\sttilt\Lambda$, which is a subset of $\wttilt\Lambda$, the theory of \emph{mutations} developed in \cite{AIR} enables us to compute the Hasse quiver of $\sttilt\Lambda$.
By restricting our attention to the hereditary case, we can generalize mutations of support $\tau$-tilting modules to wide $\tau$-tilting modules (= rigid modules), and show the following combinatorial consequence on the structure of the Hasse quiver of $\wttilt\Lambda$.
\begin{theoremi}[= Corollary \ref{cor:arrow}] \label{thm:d}
  Let $\Lambda$ be a hereditary artin algebra. Then for $M \in \wttilt\Lambda$, there are exactly $|M|$ arrows starting at $M$ in the Hasse quiver of $\wttilt\Lambda$. More precisely, for each indecomposable direct summand $X$ of $M$, we can define a mutation $\mu_X(M)$, and there is an arrow $M \to \mu_X(M)$ for each $X$.
\end{theoremi}
\begin{example}\label{ex:intro}
  Let $Q$ be a quiver $1 \ot 2$ and consider its path algebra $kQ$. Then one can compute the Hasse quivers of $\wttilt \Lambda$ and $\ice \Lambda$ as the following figures.
  \begin{figure}[hbt]
    \centering
    \begin{minipage}{0.3\textwidth}
      \centering
      \begin{tikzpicture}
        \node (1) at (-1,1) {$\sst{1}$};
        \node (2) at (1,1) {$\sst{2}$};
        \node (12) at (0,2) {$\sst{2 \\ 1}$};

        \draw[->] (1) -- (12);
        \draw[->] (12) -- (2);
        \draw[dashed] (1) -- (2);
      \end{tikzpicture}
      \caption{$\mod kQ$}
    \end{minipage}
    \begin{minipage}{0.3\textwidth}
      \centering
      \begin{tikzpicture}
        \node (10) at (5,0) {$\sst{0}$};
        \node (9) at (4.5,1) {$\sst{2}$};
        \node (8) at (4,2) {$\sst{2}\oplus\sst{2 \\ 1}$};
        \node (7) at (6,1.5) {$\sst{1}$};
        \node (6) at (5,3) {$\sst{1}\oplus\sst{2 \\ 1}$};
        \node[red] (0) at (3,1) {$\sst{2 \\ 1}$};
        \draw[->] (6) -- (7);
        \draw[->] (6) -- (8);
        \draw[->] (8) -- (9);
        \draw[->] (9) -- (10);
        \draw[->] (7) -- (10);
        \draw[->] (0) -- (10);
        \draw[->] (8) -- (0);
      \end{tikzpicture}
      \caption{$\Hasse(\wttilt kQ)$}
    \end{minipage}
    \begin{minipage}{0.3\textwidth}
      \centering
      \begin{tikzpicture}
        \node (10) at (5,0) [rectangle, rounded corners, draw] {
        \begin{tikzpicture}[scale=0.4, every node/.style={scale=0.5}]
          \node (1) at (-1,1)  {};
          \node (2) at (1,1) {};
          \node (12) at (0,2) {};
          \draw[->] (1) -- (12);
          \draw[->] (12) -- (2);
        \end{tikzpicture}
        };
        \node (9) at (4.5,1) [rectangle, rounded corners, draw] {
        \begin{tikzpicture}[scale=0.4, every node/.style={scale=0.5}]
          \node (1) at (-1,1)  {};
          \node (2) at (1,1) [blackv] {};
          \node (12) at (0,2)  {};
          \draw[->] (1) -- (12);
          \draw[->] (12) -- (2);
        \end{tikzpicture}
        };
        \node (8) at (4,2) [rectangle, rounded corners, draw] {
        \begin{tikzpicture}[scale=0.4, every node/.style={scale=0.5}]
          \node (1) at (-1,1)  {};
          \node (2) at (1,1) [blackv] {};
          \node (12) at (0,2) [blackv] {};

          \draw[->] (1) -- (12);
          \draw[->] (12) -- (2);
        \end{tikzpicture}
        };
        \node (7) at (6,1.5) [rectangle, rounded corners, draw] {
        \begin{tikzpicture}[scale=0.4, every node/.style={scale=0.5}]
          \node (1) at (-1,1) [blackv] {};
          \node (2) at (1,1)  {};
          \node (12) at (0,2)  {};

          \draw[->] (1) -- (12);
          \draw[->] (12) -- (2);
        \end{tikzpicture}
        };
        \node (6) at (5,3) [rectangle, rounded corners, draw] {
        \begin{tikzpicture}[scale=0.4, every node/.style={scale=0.5}]
          \node (1) at (-1,1) [blackv] {};
          \node (2) at (1,1) [whitev] {};
          \node (12) at (0,2) [blackv] {};

          \draw[->] (1) -- (12);
          \draw[->] (12) -- (2);
        \end{tikzpicture}
        };
        \node[red] (0) at (3,1) [rectangle, rounded corners, draw] {
        \begin{tikzpicture}[scale=0.4, every node/.style={scale=0.5}]
          \node (1) at (-1,1)  {};
          \node (2) at (1,1) {};
          \node (12) at (0,2) [blackv] {};

          \draw[black, ->] (1) -- (12);
          \draw[black, ->] (12) -- (2);
        \end{tikzpicture}
        };

        \draw[->] (6) -- (7);
        \draw[->] (6) -- (8);
        \draw[->] (8) -- (9);
        \draw[->] (9) -- (10);
        \draw[->] (7) -- (10);
        \draw[->] (0) -- (10);
        \draw[->] (8) -- (0);
      \end{tikzpicture}
      \caption{$\Hasse(\ice kQ)$}
    \end{minipage}
  \end{figure}
  The left figure is the Auslander-Reiten quiver of $\mod kQ$, and the middle and the right ones are the Hasse quivers of $\wttilt kQ$ and $\ice kQ$ respectively. In the right figure, the black vertices correspond to wide $\tau$-tilting modules and the white are the rest.
  Observe that $\Hasse(\wttilt kQ)$ contains $\Hasse(\sttilt kQ)$ as a full subquiver, which is always the case (Proposition \ref{prop:asfullsub}). We have a new wide $\tau$-tilting module $\sst{2\\1}$ which is not a support $\tau$-tilting. This corresponds to an ICE-closed subcategory $\add \sst{2\\1}$, which is a wide subcategory but not a torsion class.
  There are exactly two arrows starting at $\sst{2} \oplus \sst{2\\1}$, and we have $\mu_{\sst{2}}(\sst{2} \oplus \sst{2\\1}) = \sst{2\\1}$ and $\mu_{\sst{2\\1}}(\sst{2} \oplus \sst{2\\1}) = \sst{2}$.
\end{example}
For more examples of $\Hasse(\wttilt \Lambda)$, see Section \ref{sec:ex}.
The statement of Theorem \ref{thm:d} makes sense for the non-hereditary case, and some examples of $\Hasse(\wttilt \Lambda)$ suggest that there should be more classes of algebras which satisfy Theorem \ref{thm:d}, see Question \ref{question}.

\begin{org}
This paper is organized as follows.
In Section \ref{sec:2}, we collect basic properties of hearts of intervals which we use throughout this paper.
In Section \ref{sec:3}, we give a proof of Theorem \ref{thm:b}.
In Section \ref{sec:4}, we focus on artin algebras and study wide $\tau$-tilting modules and establish a bijection in Theorem \ref{thm:c}. Then we discuss how to obtain wide $\tau$-tilting modules from support $\tau$-tilting modules.
In Section \ref{sec:5}, we focus on hereditary artin algebras and develop the theory of mutations of rigid modules.
In Section \ref{sec:6}, we give some examples of computations of ICE-closed subcategories and wide $\tau$-tilting modules, and the Hasse quivers of ICE-closed subcategories.
\end{org}

\begin{conv}
  Throughout this paper, we assume that \emph{all categories are skeletally small}, that is, the isomorphism classes of objects form a set. In addition, \emph{$\AA$ always denotes an abelian length category}, that is, an abelian category in which every object has finite length. We also assume that \emph{all subcategories are full and closed under isomorphisms}.
\end{conv}

\section{Preliminaries on hearts of intervals}\label{sec:2}
We begin by introducing basic definitions and notations in an abelian length category $\AA$.
\begin{definition}\label{def:basicdef}
  Let $\AA$ be an abelian length category and $\CC$ a subcategory of $\AA$.
  \begin{enumerate}
    \item $\CC$ is \emph{closed under extensions} if, for any short exact sequence in $\AA$
    \[
    \begin{tikzcd}
      0 \rar & L \rar & M \rar & N \rar & 0,
    \end{tikzcd}
    \]
    we have that $L,N \in \CC$ implies $M \in \CC$.
    \item $\CC$ is \emph{closed under quotients (resp. subobjects) in $\AA$} if, for every object $C \in \CC$, any quotients (resp. subobjects) of $C$ in $\AA$ belong to $\CC$.
    \item $\CC$ is a \emph{torsion class (resp. torsion-free class) in $\AA$} if $\CC$ is closed under extensions and quotients in $\AA$ (resp. extensions and subobjects).
    \item $\CC$ is closed under \emph{images (resp. kernels, cokernels)} if, for every map $\varphi \colon C_1 \to C_2$ with $C_1, C_2 \in \CC$, we have $\im\varphi \in \CC$ (resp. $\ker\varphi\in\CC$, $\coker\varphi\in\CC$).
    \item $\CC$ is a \emph{wide subcategory of $\AA$} if $\CC$ is closed under extensions, kernels and cokernels.
    \item $\CC$ is an \emph{ICE-closed subcategory of $\AA$} if $\CC$ is closed under images, cokernels and extensions.
  \end{enumerate}
\end{definition}
Every wide subcategory $\WW$ of $\AA$ becomes an abelian length category, and we always regard $\WW$ as such.
Both wide subcategories and torsion classes are clearly ICE-closed subcategories. Moreover, every torsion class in a wide subcategory (viewed as an abelian category) is ICE-closed:
\begin{lemma}\label{lem:torsice}
  Let $\WW$ be a wide subcategory of $\AA$ and $\CC$ a torsion class in $\WW$. Then $\CC$ is an ICE-closed subcategory of $\AA$.
\end{lemma}
\begin{proof}
  Obviously $\CC$ is extension-closed. Take any map $\varphi \colon C \to C'$ with $C,C' \in \CC$. Then since $\WW$ is wide, $\ker\varphi$, $\coker\varphi$ and $\im\varphi$ belong to $\WW$. Therefore, we obtain the following short exact sequences in $\WW$:
  \[
  \begin{tikzcd}[row sep=0]
    0 \rar & \ker\varphi \rar & C \rar & \im\varphi \rar & 0, \\
    0 \rar & \im\varphi \rar & C' \rar & \coker\varphi \rar & 0.
  \end{tikzcd}
  \]
  Then since $\CC$ is closed under quotients in $\WW$, we obtain $\im\varphi, \coker\varphi \in \CC$. Therefore, $\CC$ is closed under images and cokernels in $\AA$.
\end{proof}
We will see later in Corollary \ref{cor:widetors} that the converse holds, that is, every ICE-closed subcategory is a torsion class in some wide subcategory.

Throughout this paper, we will use the following notation.
\begin{itemize}
  \item $\tors\AA$ denotes the set of torsion classes in $\AA$.
  \item $\torf\AA$ denotes the set of torsion-free classes in $\AA$.
  \item $\wide\AA$ denotes the set of wide subcategories of $\AA$.
  \item $\ice\AA$ denotes the set of ICE-closed subcategories of $\AA$.
\end{itemize}
These sets are partially ordered by inclusion, and clearly closed under arbitrary (set-theoretic) intersections, thus they are complete meet-semilattices. It is well-known that each complete meet semi-lattice is a complete lattice, hence these four sets become complete lattices. We denote by $\vee$ the join in these lattices.

For a collection $\CC$ of objects in $\AA$, we define the following full subcategories of $\AA$.
\begin{itemize}
  \item $\add\CC$ denotes the subcategory consisting of direct summands of finite direct sums of objects in $\CC$.
  \item $\Fac\CC$ denotes the subcategory consisting of quotients of objects in $\add\CC$.
  \item $\Sub\CC$ denotes the subcategory consisting of subobjects of objects in $\add\CC$.
  \item $\Filt\CC$ denotes the smallest extension-closed subcategory of $\AA$ containing $\CC$, or equivalently, the subcategory consisting of $M \in \AA$ such that there is a sequence of subobjects of $M$
  \[
  0 = M_0 < M_1 < \cdots < M_n = M
  \]
  satisfying $M_i/M_{i-1} \in \CC$ for each $i$.
  \item For two collections $\CC$ and $\DD$ of objects in $\AA$, we denote by $\CC * \DD$ the subcategory of $\AA$ consisting of $M \in \AA$ such that there is a short exact sequence
  \[
  \begin{tikzcd}
    0 \rar & C \rar & M \rar & D \rar & 0
  \end{tikzcd}
  \]
  with $C \in \CC$ and $D \in \DD$.
  \item $\TTT(\CC)$ (resp. $\FFF(\CC)$) denotes the smallest torsion class containing $\CC$ (resp. the smallest torsion-free class containing $\CC$).
  \item $\CC^\perp$ (resp. $^\perp\CC$) denotes the subcategory consisting of $M \in \AA$ satisfying $\AA(\CC,M) = 0$ (resp. $\AA(M,\CC) = 0$).
\end{itemize}

The following are well-known properties of torsion(-free) classes in $\AA$.
\begin{proposition}\label{prop:torsbasic}
  Let $\AA$ be an abelian length category and $\CC$ a subcategory of $\AA$.
  \begin{enumerate}
    \item $\AA = \TT * \TT^\perp$ and $\AA = {}^\perp\FF * \FF$ hold for $\TT \in \tors\AA$ and $\FF \in \torf\AA$.
    \item \cite[Lemma 3.1]{MS}, \cite[Lemma 3.11]{enomono}
    $\TTT(\CC) = {}^\perp (\CC^\perp) = \Filt(\Fac\CC)$ and $\FFF(\CC) = ({}^\perp\CC)^\perp = \Filt(\Sub\CC)$ hold.
    \item $(-)^\perp \colon \tors\AA \to \torf\AA$ and $^\perp (-) \colon \torf\AA \to \tors\AA$ are mutually inverse anti-isomorphisms of complete lattices.
  \end{enumerate}
\end{proposition}

For a poset $P$ and $a,b \in P$ with $a \leq b$, we define $[a,b]:= \{ c \in P \, | \, a \leq c \leq b \}$, which is a full subposet of $P$. We call it the \emph{interval in $P$}.
For a poset $P$, its \emph{Hasse quiver $\Hasse P$} is a quiver defined as follows.
\begin{itemize}
  \item The vertex set of $\Hasse P$ is $P$.
  \item We draw an arrow $p \to q$ if $p$ covers $q$ in $P$, that is, $p > q$ holds and there is no $r\in P$ with $p > r > q$.
\end{itemize}

To each interval in $\tors\AA$ or $\torf\AA$, we can associate a subcategory of $\AA$ which is in some sense the ``difference." This construction was introduced and used by several papers such as \cite{jasso, DIRRT, AP, tattar}. In this paper, following \cite{tattar}, we call this construction the \emph{heart} of the interval.
\begin{definition}
  Let $[\GG,\FF]$ be an interval in $\torf\AA$.
  \begin{enumerate}
    \item The \emph{heart} of the interval $[\GG,\FF]$ in $\torf\AA$ is a subcategory of $\AA$ defined by
    \[
    \HH_{[\GG,\FF]} := \FF \cap {}^\perp \GG.
    \]
    \item The interval $[\GG,\FF]$ in $\torf\AA$ is called a \emph{wide interval} (resp. \emph{ICE interval}) if its heart is a wide subcategory of $\AA$ (resp. an ICE-closed subcategory of $\AA$).
  \end{enumerate}
\end{definition}
\begin{remark}\label{rem:heart}
  There is a dual construction, which associates to an interval $[\UU,\TT]$ in $\tors\AA$ the subcategory $\HH_{[\UU,\TT]}:= \TT \cap \UU^\perp$.
  This construction for torsion classes is more common in the past literature, and these two are dual to each other via the anti-isomorphism $\tors\AA \rightleftarrows \torf\AA$ in Proposition \ref{prop:torsbasic} (3), that is, we have $\HH_{[\UU,\TT]} = \HH_{[\TT^\perp,\UU^\perp]}$.
  When investigating ICE-closed subcategories, using torsion-free classes is more suitable than torsion classes since we have to deal with $\FFF(\CC)$ and $\WR(\FFF(\CC))$, see Corollary \ref{cor:widetors} for example.
  On the other hand, it is more appropriate to work on torsion classes when studying wide $\tau$-tilting modules. Thus in this paper, \emph{we mainly use intervals in $\torf\AA$ in Sections \ref{sec:2} and \ref{sec:3}, and intervals in $\tors\AA$ in Sections \ref{sec:4} and \ref{sec:6}}.
\end{remark}
Every torsion class and torsion-free class can be realized as a heart of intervals in $\torf\AA$. Indeed, we can easily check the following by definition and Proposition \ref{prop:torsbasic}.
\begin{lemma}\label{lem:torsint}
  Let $\AA$ be an abelian length category and $\FF$ a torsion-free class in $\AA$.
  \begin{enumerate}
    \item $\HH_{[0,\FF]} = \FF$ holds, thus it is a torsion-free class in $\AA$.
    \item $\HH_{[\FF,\AA]} = {}^\perp \FF$ holds, thus it is a torsion class in $\AA$. In particular, every torsion class $\TT$ in $\AA$ is a heart of $[\TT^\perp,\AA]$ in $\torf\AA$.
    \item The heart of every interval in $\torf\AA$ is closed under extensions and images.
  \end{enumerate}
\end{lemma}

The following useful result was established in \cite{AP}, which gives a characterization and a reduction of wide intervals.
\begin{proposition}\cite[Theorems 4.2, 5.2]{AP}\label{prop:wideint}
  Let $[\GG,\FF]$ be an interval in $\torf\AA$.
  \begin{enumerate}
    \item $[\GG,\FF]$ is a wide interval if and only if $\GG = \FF \cap \bigcap_{i \in I} \FF_i$ holds for some $\FF_i$'s such that there is an arrow $\FF \to \FF_i$ in $\Hasse(\torf\AA)$ for each $i \in I$.
    \item Suppose that $[\GG,\FF]$ is a wide interval. Then we have isomorphisms of lattices:
    \[
    \begin{tikzcd}[column sep = large]
      {[\GG,\FF]} \rar["{(-)\cap {}^\perp\GG}", shift left] & \torf\HH_{[\GG,\FF]} \lar["{(-)*\GG}", shift left].
    \end{tikzcd}
    \]
  \end{enumerate}
\end{proposition}

As for hearts, we have the following basic equalities.
\begin{lemma}\label{lem:interval}
  Let $[\GG,\FF]$ be an interval in $\torf\AA$. Then the following equalities hold:
  \begin{enumerate}
    \item $\FF = \HH_{[\GG,\FF]} * \GG$.
    \item $\GG = \FF \cap \HH_{[\GG,\FF]}^\perp$.
  \end{enumerate}
\end{lemma}
\begin{proof}
  Although some variants of this lemma seem to be well-known (e.g. \cite[Theorem 3.4]{DIRRT}, \cite[Lemma 3.1]{AP}), we give a proof here for the convenience of the reader.

  (1)
  Since $\FF$ is extension-closed, we have $\HH_{[\GG,\FF]} * \GG \subseteq \FF * \GG \subseteq \FF$. Conversely, let $F \in \FF$. Since $\AA = {}^\perp\GG * \GG$ by Proposition \ref{prop:torsbasic}, we have the following short exact sequence
  \[
  \begin{tikzcd}
    0 \rar & L \rar & F \rar & N \rar & 0
  \end{tikzcd}
  \]
  with $L \in {}^\perp\GG$ and $N \in \GG$. Then we have $L \in \FF$ since $\FF$ is closed under subobjects. Thus $L \in \HH_{[\GG,\FF]}$, and we obtain $F \in \HH_{[\GG,\FF]} * \GG$.

  (2)
  Since $\GG \subseteq \FF$ and $\HH_{[\GG,\FF]} \subseteq {}^\perp\GG$, clearly we have $\GG \subseteq \FF \cap \HH_{[\GG,\FF]}^\perp$. Conversely, let $F \in \FF \cap \HH_{[\GG,\FF]}^\perp$. Then by (1), we have the following short exact sequence
  \[
  \begin{tikzcd}
    0 \rar & L \rar["\iota"] & F \rar & N \rar & 0
  \end{tikzcd}
  \]
  with $L \in \HH_{[\GG,\FF]}$ and $N \in \GG$. Then $F \in \HH_{[\GG,\FF]}^\perp$ implies $\iota = 0$, thus $F \iso N \in \GG$.
\end{proof}

We will use the following invariance property of hearts.
\begin{lemma}\label{lem:preserveheart}
  Let $[\GG,\FF]$ be a wide interval in $\torf\AA$. Then the bijection $(-)\cap {}^\perp\GG \colon [\GG,\FF] \xrightarrow{\sim} \torf\HH_{[\GG,\FF]}$ in Proposition \ref{prop:wideint} preserves hearts, that is,
  for every interval $[\FF_1,\FF_2]$ in $[\GG,\FF]$,
  \[
  \HH_{[\FF_1,\FF_2]} = \HH_{[\FF_1 \cap {}^\perp\GG, \FF_2 \cap {}^\perp\GG]}
  \]
  holds, where $\HH$ in the right hand side is considered in an abelian category $\HH_{[\GG,\FF]}$.
\end{lemma}
\begin{proof}
  The proof is essentially the same as \cite[Theorem 4.2 (2)]{AP}, but we provide a proof for completeness.
  Let $M$ be in $\HH_{[\FF_1,\FF_2]}$.
  Then since $\GG \subseteq \FF_1$, $\FF_1 \cap {}^\perp\GG \subseteq \FF_1$ and $M \in {}^\perp\FF_1$,
  we have $M \in {}^\perp\GG$ and $M \in {}^\perp(\FF_1 \cap {}^\perp\GG)$.
  Thus $M$ belongs to ${}^\perp(\FF_1\cap {}^\perp\GG) \cap (\FF_2 \cap {}^\perp\GG) = \HH_{[\FF_1 \cap {}^\perp\GG, \FF_2 \cap {}^\perp\GG]}$.

  Conversely, let $M$ be in $\HH_{[\FF_1 \cap {}^\perp\GG, \FF_2 \cap {}^\perp\GG]}$. Then we have $M \in \FF_2 \cap {}^\perp\GG \subseteq \FF_2$, thus it suffices to show $M \in {}^\perp\FF_1$. Take any map $\varphi \colon M \to F_1$ with $F_1 \in \FF_1$.
  Then we obtain the following exact commutative diagram with $T \in \HH_{[\GG,\FF_1]}$ and $G \in \GG$ by applying Lemma \ref{lem:interval} to $[\GG,\FF_1]$.
  \[
  \begin{tikzcd}
    & & M \dar["\varphi"] \ar[dl, dashed, "\ov{\varphi}"']\ar[rd, "0"]\\
    0 \rar & T \rar["\iota"'] & F_1 \rar["\pi"'] & G \rar & 0
  \end{tikzcd}
  \]
  Since $M \in {}^\perp \GG$ and $G \in \GG$, we have $\pi\varphi = 0$. Thus we have a map $\ov{\varphi}$ which makes the above diagram commute.
  On the other hand, we have $T \in \HH_{[\GG,\FF_1]} = \FF_1 \cap {}^\perp\GG$. Thus by $M \in {}^\perp (\FF_1 \cap {}^\perp\GG)$, we obtain $\ov{\varphi}= 0$. Therefore we have $\varphi = 0$.
\end{proof}

\section{Characterization of ICE intervals}\label{sec:3}
First, we will show that every ICE-closed subcategory can be realized as a heart of some interval in $\torf\AA$ (or $\tors\AA$). This observation is the starting point of this paper.
\begin{proposition}\label{prop:arashi}
  Let $\CC$ be an ICE-closed subcategory of an abelian length category $\AA$.
  Then $\CC$ is equal to the heart of $[\FFF(\CC) \cap \CC^\perp, \FFF(\CC)]$ in $\torf\AA$, that is, $\CC =\FFF(\CC)\cap {}^\perp(\FFF(\CC) \cap \CC^\perp)$ holds.
\end{proposition}
\begin{proof}
  First, note that we have ${}^\perp(\FFF(\CC) \cap \CC^\perp) = {}^\perp\CC \vee \TTT(\CC) = \TTT({}^\perp\CC \cup \CC)$ by Proposition \ref{prop:torsbasic}.
  Hence it suffices to show $\CC = \FFF(\CC) \cap \TTT({}^\perp\CC \cup \CC)$. Clearly $\CC \subseteq\FFF(\CC) \cap \TTT({}^\perp\CC \cup \CC)$ holds, hence it is enough to prove the converse inclusion.
  We divide the proof into three steps.

  \un{(Step 1): $\CC \supseteq \Sub\CC\cap\Fac(^{\perp}\CC\cup\CC)$}. Let $X$ be an object in $\Sub\CC\cap\Fac(^{\perp}\CC\cup\CC)$.
  Then there exist a surjection $Y \oplus C \defl X$ with $Y \in {}^\perp \CC$ and $C \in \CC$, and an injection $X \hookrightarrow C'$ with $C' \in \CC$. Thus $X$ is the image of the composition $\varphi \colon Y\oplus C \defl X \hookrightarrow C'$. Since $\AA(Y,C') = 0$, the image of $\varphi$ is equal to that of the composition of $\varphi$ and the inclusion $C \hookrightarrow Y\oplus C$.
  Thus $X$ is the image of $C \to C'$. Since $\CC$ is closed under images, $X$ is in $\CC$.

  \un{(Step 2): $\CC \supseteq \Sub\CC\cap \TTT(^{\perp}\CC\cup\CC)$}. Let $X$ be in $\Sub\CC\cap \TTT({}^\perp\CC\cup\CC)$.
  Recall that $\TTT({}^\perp \CC \cup \CC) = \Filt(\Fac({}^\perp \CC \cup \CC))$ by Proposition \ref{prop:torsbasic} (2).
  We will show $X \in \CC$ by induction on the $\Fac(^{\perp}\CC\cup\CC)$-filtration length $n$ of $X$. If $n=1$, then this follows from (Step 1).
  Suppose $n>1$. There is a short exact sequence
    \[
    \begin{tikzcd}
      0 \rar & Y \rar & X \rar & Z \rar & 0,
    \end{tikzcd}
    \]
  where $Z$ is in $\Fac({}^\perp \CC\cup\CC)$ and the $\Fac(^{\perp}\CC\cup\CC)$-filtration length of $Y$ is smaller than $n$. Since $X$ is in $\Sub\CC$, so is $Y$. By the induction hypothesis, we have $Y\in\CC$. Since $X$ is in $\Sub\CC$, there is an injection $X \hookrightarrow C$ with $C \in\CC$. Then we obtain the following diagram
    \[
    \begin{tikzcd}
      0 \rar & Y \rar \dar[equal] & X \dar[hookrightarrow] \ar[rd, phantom, "{\rm p.o.}"] \rar & Z \rar \dar[hookrightarrow] & 0 \\
      0 \rar & Y \rar & C \rar & C/Y \rar & 0
    \end{tikzcd}
    \]
  where the right square is a pushout diagram. Since $\CC$ is closed under cokernels, we have $C/Y\in\CC$. Then $Z$ is in $\Sub\CC$, thus in $\CC$ by (Step 1). Now $X \in \CC$ holds since $\CC$ is closed under extensions.

  \un{(Step 3): $\CC \supseteq \FFF(\CC)\cap \TTT(^{\perp}\CC\cup\CC)$}. Let $X$ be in $\FFF(\CC)\cap \TTT(^{\perp}\CC\cup\CC)$.
  Recall that $\FFF(\CC) = \Filt(\Sub\CC)$ holds by Proposition \ref{prop:torsbasic} (2).
  We show $X \in \CC$ by the induction on the $\Sub\CC$-filtration length $n$ of $X$. If $n=1$, then this follows from (Step 2).
  Suppose $n>1$. There is a short exact sequence
    \[
    \begin{tikzcd}
      0 \rar & Y \rar & X \rar & Z \rar & 0,
    \end{tikzcd}
    \]
  where $Y$ is in $\Sub\CC$ and the $\Sub\CC$-filtration length of $Z$ is smaller than $n$. Since $X$ is in $\TTT(^{\perp}\CC\cup\CC)$, so is $Z$. By the induction hypothesis, we have $Z\in\CC$. Since $Y$ is in $\Sub\CC$, there is an injection $Y \hookrightarrow C$ with $C\in \CC$. Then we can take the following pushout diagram.
    \[
    \begin{tikzcd}
      0 \rar & Y \ar[rd, phantom, "{\rm p.o.}"] \rar \dar[hookrightarrow] & X \dar[hookrightarrow] \rar & Z \rar \dar[equal] & 0 \\
      0 \rar & C \rar & E \rar & Z \rar & 0
    \end{tikzcd}
    \]
  Since $\CC$ is extension-closed, we have $E\in\CC$, hence $X$ is in $\Sub\CC$. By (Step 2), we have $X\in\CC$.
\end{proof}

Next, we will characterize ICE intervals in $\torf\AA$ by using wide intervals. We will use the following construction of wide subcategories and its relation to wide intervals.
\begin{definition}\cite{IT,MS}
  Let $\FF$ be a torsion-free class in $\AA$. Then we define the subcategory $\WR(\FF)$ of $\AA$ as follows:
  \[
  \WR(\FF):= \{ X \in \FF \, | \, \text{for every map $\varphi \colon X \to F$ with $F \in \FF$, we have $\coker\varphi \in \FF$} \}.
  \]
\end{definition}
It was shown in \cite{MS} that $\WR(\FF)$ is a wide subcategory of $\AA$ (see also \cite[Theorem 4.5]{enomono}). Moreover, the following result is implicitly shown in \cite{AP}.
\begin{proposition}\label{prop:torfwide}
  Let $\FF$ be a torsion-free class in $\AA$. Then the following hold.
  \begin{enumerate}
    \item $\WR(\FF)$ coincides with the heart of $[\FF \cap \WR(\FF)^\perp, \FF]$.
    \item $\FF\cap \WR(\FF)^\perp = \FF \cap \bigcap \{ \FF'\in\torf\AA \, | \, \text{there is an arrow $\FF \to \FF'$ in $\Hasse(\torf\AA)$} \}$.
  \end{enumerate}
  In particular, $[\FF\cap \WR(\FF)^\perp,\FF]$ is a wide interval.
\end{proposition}
\begin{proof}
  (1)
  This immediately follows from (the torsion-free version of) \cite[Proposition 6.3]{AP}.

  (2)
  Define $\FF^-:=\FF \cap \WR(\FF)^\perp$. Then the interval $[\FF^-, \FF]$ in $\torf\AA$ is a wide interval by (1). Thus it is a meet interval in their terminology by \cite[Theorem 5.2]{AP}, that is,
  \[
  \FF^- = \FF \cap \bigcap \{ \FF'\in[\FF^-,\FF] \, | \, \text{there is an arrow $\FF \to \FF'$ in $\Hasse(\torf\AA)$} \}.
  \]
  On the other hand, by \cite[Propositions 3.2, 6.5]{AP}, we have
  \begin{align*}
  \{ \FF'\in[\FF^-,\FF] \, | \, \text{there is an arrow $\FF \to \FF'$ in $\Hasse(\torf\AA)$} \} \\
  =
  \{ \FF'\in \torf\AA \, | \, \text{there is an arrow $\FF \to \FF'$ in $\Hasse(\torf\AA)$} \}.
  \end{align*}
  Thus the assertion follows.
\end{proof}

For a fixed $\FF \in \torf\AA$, this interval $[\FF\cap \WR(\FF)^\perp,\FF]$ is the maximal wide interval of the form $[\GG,\FF]$ by Proposition \ref{prop:wideint} (1).
We define $\FF^-$ as follows:
\[
\FF^- := \FF \cap \WR(\FF)^\perp = \FF \bigcap \{ \FF'\in\torf\AA \, | \, \text{there is an arrow $\FF \to \FF'$ in $\Hasse(\torf\AA)$} \}.
\]
Note that this construction of $\FF^-$ only depends on the lattice structure of $\torf\AA$. Also note that $\FF^- =  \bigcap \{ \FF'\in\torf\AA \, | \, \text{there is an arrow $\FF \to \FF'$ in $\Hasse(\torf\AA)$} \}$ if there is at least one arrow in $\Hasse(\torf\AA)$ starting at $\FF$.
Now we are ready to state purely lattice-theoretical characterization of ICE intervals.
\begin{theorem}\label{thm:iceintchar}
  Let $[\GG,\FF]$ be an interval in $\torf\AA$. Then the following are equivalent.
  \begin{enumerate}
    \item $\FF^- \subseteq \GG \subseteq \FF$ holds.
    \item There is a torsion-free class $\GG'$ with $\GG' \subseteq \GG \subseteq \FF$ and $[\GG',\FF]$ is a wide interval.
    \item $[\GG,\FF]$ is an ICE interval.
  \end{enumerate}
  In the situation of {\upshape (2)}, we have that $\HH_{[\GG,\FF]}$ is a torsion class in $\HH_{[\GG',\FF]}$.
  In particular, if $[\GG,\FF]$ is an ICE interval, then $\HH_{[\GG,\FF]}$ is a torsion class in $\WR(\FF)$.
\end{theorem}
\begin{proof}
  (1) $\Rightarrow$ (2):
  $\GG':= \FF^-$ satisfies the condition by Proposition \ref{prop:torfwide}.

  (2) $\Rightarrow$ (3):
  Put $\WW:=\HH_{[\GG',\FF]}$ for simplicity, which is a wide subcategory of $\AA$.
  Since $[\GG',\FF]$ is a wide interval, we have an isomorphism of posets
  \[
  \begin{tikzcd}[column sep = large]
    {[\GG',\FF]} \rar["\Phi"] & \torf\WW
  \end{tikzcd}
  \]
  defined as $\Phi(-) = (-) \cap {}^\perp \GG'$ by Proposition \ref{prop:wideint} (2).
  Moreover, $\Phi$ preserves hearts by Lemma \ref{lem:preserveheart}.
  Therefore, we have $\HH_{[\GG,\FF]} = \HH_{[\Phi(\GG),\Phi(\FF)]}$ by $\GG' \subseteq \GG \subseteq \FF$, where the second $\HH$ is considered inside $\WW$.
  Since $\Phi(\FF) = \FF \cap {}^\perp\GG' = \HH_{[\GG',\FF]} = \WW$ holds, $\HH_{[\GG,\FF]}$ coincides with $\HH_{[\Phi(\GG),\WW]}$, which is a torsion class in $\WW$ by Lemma \ref{lem:torsint}.
  Thus it is an ICE-closed subcategory of $\AA$ by Lemma \ref{lem:torsice}.

  (3) $\Rightarrow$ (1):
  Suppose that $[\GG,\FF]$ is an ICE interval, and put $\WW:=\WR(\FF)$.
  First, we will show the following claim.

  {\bf (Claim): $\HH_{[\GG,\FF]} \subseteq \WW$.}

  \emph{Proof of the claim.}
  Let $X$ be an object in $\HH_{[\GG,\FF]}$. We have $X \in \HH_{[\GG,\FF]} \subseteq \FF$, thus in order to show $X \in \WW$, it suffices to show that every map $\varphi\colon X \to F$ with $F \in \FF$ satisfies $\coker\varphi\in\FF$.
  By Lemma \ref{lem:interval} we have $\FF = \HH_{[\GG,\FF]} * \GG$, thus we obtain the following short exact sequence
  \[
  \begin{tikzcd}
    & & X \dar["\varphi"] \\
    0 \rar & H \rar["\iota"] & F \rar["\pi"] & G \rar & 0
  \end{tikzcd}
  \]
  with $H \in \HH_{[\GG,\FF]}$ and $G \in \GG$. Then since $X \in \HH_{[\GG,\FF]} = \FF \cap {}^\perp\GG$, we have $\pi \varphi = 0$. Thus $\varphi$ factors through $\iota$, and we obtain the following commutative diagram.
  \[
  \begin{tikzcd}
    & X \rar[equal] \dar["\ov{\varphi}"']& X \dar["\varphi"] \\
    0 \rar & H \ar[rd, phantom, "{\rm p.o.}"] \rar["\iota"]\dar & F \rar["\pi"] \dar & G \dar[equal]\rar & 0 \\
    0 \rar & \coker \ov{\varphi} \dar \rar & \coker\varphi \dar \rar & G \rar & 0 \\
    & 0 & 0
  \end{tikzcd}
  \]
  The bottom left square is checked to be a pushout square, hence the bottom sequence is a short exact sequence.
  Since $\HH_{[\GG,\FF]}$ is closed under cokernels, we have $\coker\ov{\varphi} \in \HH_{[\GG,\FF]} \subseteq \FF$. Then we obtain $\coker\varphi \in \FF$ since $G \in \GG \subseteq \FF$ and $\FF$ is extension-closed. $\qedb$

  By the claim, we have $\WW^\perp \subseteq \HH_{[\GG,\FF]}^\perp$, which implies $\FF^- = \FF\cap\WW^\perp \subseteq \FF \cap \HH_{[\GG,\FF]}^\perp = \GG$ (the last equality follows from Lemma \ref{lem:interval}). Thus we obtain (1).

  The remaining assertions follow from the proof of (2) $\Rightarrow$ (3) and Proposition \ref{prop:torfwide}.
\end{proof}
As an immediate corollary, we obtain one of the main results in this paper:
\begin{corollary}\label{cor:widetors}
  Let $\CC$ be a subcategory of $\AA$. Then the following are equivalent.
  \begin{enumerate}
    \item $\CC$ is an ICE-closed subcategory of $\AA$.
    \item $\CC$ is a torsion class in a wide subcategory $\WR(\FFF(\CC))$.
    \item $\CC$ is a torsion class in some wide subcategory of $\AA$.
  \end{enumerate}
\end{corollary}
\begin{proof}
  (1) $\Rightarrow$ (2):
  Let $\CC$ be an ICE-closed subcategory. By Proposition \ref{prop:arashi}, it is a heart of some interval in $\torf\AA$ of the form $[\GG,\FFF(\CC)]$.
  Then Theorem \ref{thm:iceintchar} shows that $\CC$ is a torsion class in $\WR(\FFF(\CC))$.

  (2) $\Rightarrow$ (3): Trivial.

  (3) $\Rightarrow$ (1): This is Lemma \ref{lem:torsice}.
\end{proof}
The construction of a wide subcategory $\WR(\FFF(\CC))$ from an ICE-closed subcategory $\CC$ is somewhat technical, but we will see later in Theorem \ref{thm:classifyice} that it coincides with the smallest wide subcategory containing $\CC$.

Theorem \ref{thm:iceintchar} together with Proposition \ref{prop:arashi} gives a way to realize all ICE-closed subcategories as hearts if the classification of all torsion(-free) classes is given:
\begin{corollary}\label{cor:ice}
  All ICE-closed subcategories of $\AA$ can be obtained by the following process.
  \begin{itemize}
    \item Choose a torsion-free class $\FF$ in $\AA$.
    \item Compute $\FF^- := \FF \bigcap \{ \FF'\in\torf\AA \, | \, \text{there is an arrow $\FF \to \FF'$ in $\Hasse(\torf\AA)$} \}$.
    \item For each $\GG \in [\FF^-,\FF]$, the heart $\HH_{[\GG,\FF]} = \FF \cap {}^\perp\GG$ is an ICE-closed subcategory.
  \end{itemize}
  Dually, all ICE-closed subcategories of $\AA$ can be obtained by the following process.
  \begin{itemize}
    \item Choose a torsion class $\UU$ in $\AA$.
    \item Compute $\UU^+ := \UU \bigvee \{ \UU'\in\tors\AA \, | \, \text{there is an arrow $\UU \ot \UU'$ in $\Hasse(\tors\AA)$} \}$.
    \item For each $\TT \in [\UU,\UU^+]$, the heart $\HH_{[\UU,\TT]} = \TT \cap \UU^\perp$ is an ICE-closed subcategory.
  \end{itemize}
\end{corollary}
\begin{proof}
  The former half is clear from Theorem \ref{thm:iceintchar} and Proposition \ref{prop:arashi}.
  The latter half follows from the anti-isomorphisms $(-)^\perp \colon \tors\AA \rightleftarrows \torf\AA \colon {}^\perp(-)$ of posets in Proposition \ref{prop:torsbasic}, which preserve hearts by Remark \ref{rem:heart}.
\end{proof}

We end this section with examples of computations of ICE intervals for a finite-dimensional algebra. In the rest of this section, $k$ denotes an algebraically closed field.
For a finite-dimensional $k$-algebra $\Lambda$, we denote by $\mod\Lambda$ the category of finitely generated right $\Lambda$-modules, and we simply write $\tors\Lambda$ instead of $\tors(\mod\Lambda)$.

For actual computations, the \emph{brick labeling} of $\tors\AA$ introduced in \cite{DIRRT} is useful and intuitive. Recall that a \emph{brick $B$} in $\AA$ is an object such that $\End_\AA(B)$ is a division ring.
\begin{proposition}[{\cite[Theorems 3.3, 3.4]{DIRRT}}]\label{prop:label}
  Let $q \colon \TT \to \UU$ be an arrow in $\Hasse (\tors\AA)$. Then there is a unique brick $B_q$ in $\HH_{[\UU,\TT]}$ up to isomorphism, and $\HH_{[\UU,\TT]} = \Filt B_q$ holds.
\end{proposition}
In the above situation, we call $B_q$ the \emph{brick label} of $q$.
If there is a path from $\TT$ to $\UU$, then we have the following description of $\HH_{[\UU,\TT]}$ using the brick labeling.
\begin{proposition}\label{prop:brickseq}
  Let $\UU \xleftarrow{B_1} \bullet \xleftarrow{B_2} \cdots \bullet \xleftarrow{B_n} \TT$ be a path in $\Hasse(\tors\AA)$, where each $B_i$ is the brick label. Then we have
  \[
  \HH_{[\UU,\TT]} = (\Filt B_1) * (\Filt B_2) * \cdots * (\Filt B_n) = \Filt \{ B_1,\dots,B_n\}.
  \]
\end{proposition}
\begin{proof}
  For $\TT_1 \subseteq \TT_2 \subseteq \TT_3$ in $\tors\AA$, one can easily check that $\HH_{[\TT_1,\TT_3]} = \HH_{[\TT_1,\TT_2]} * \HH_{[\TT_2,\TT_3]}$ in the same way as in Lemma \ref{lem:interval} (1). Then the assertion inductively follows from this.
  See also \cite[Theorem 3.5]{enobinv} for the related argument.
\end{proof}
\begin{example}\label{ex:a2}
  Let $Q$ be a quiver $1 \ot 2$ and consider its path algebra $kQ$ as in Example \ref{ex:intro}. Then the following is $\Hasse(\tors kQ)$ together with its brick labeling.
  
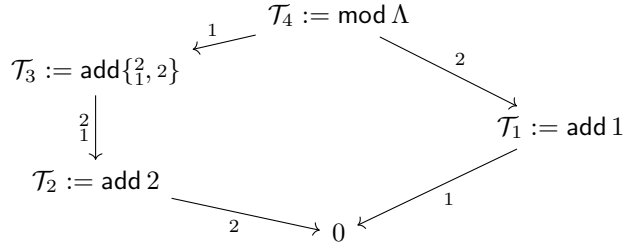
\[
  \begin{tikzcd}[row sep=tiny]
    &\TT_4:=\mod\Lambda \ar[dl, "1"'] \ar[ddr, "2"] \\
    \TT_3:=\add \{ \sst{2\\1}, \sst{2} \} \ar[dd, "\sst{2\\1}"'] \\
    & & \TT_1:=\add 1 \ar[ddl, "1"]\\
    \TT_2:=\add 2 \ar[dr, "2"'] \\
    & 0
  \end{tikzcd}
  \]
  By Theorem \ref{thm:iceintchar} (more precisely, its dual Theorem \ref{thm:b}), we have that $[0,\TT_2]$ and $[\TT_1,\TT_4]$ are both ICE intervals,
  but Proposition \ref{prop:brickseq} shows that their hearts are $\Filt 2 = \add 2$. Thus different intervals may give the same ICE-closed subcategory.
  Using Corollary \ref{cor:ice} and Proposition \ref{prop:brickseq}, we obtain the following list of ICE intervals and their hearts.
  \begin{table}[htp]
    \begin{tabular}{C|C}
      \text{ICE intervals in $\tors kQ$} & \text{ICE-closed subcategories} \\ \hline \hline
      [\TT,\TT] \text{ for every $\TT \in \tors kQ$} & 0 \\ \hline
      [0,\TT_1],[\TT_3,\TT_4] & \Filt 1 = \TT_1 \\ \hline
      [0,\TT_2], [\TT_1,\TT_4] & \Filt 2 = \TT_2  \\ \hline
      [0,\TT_3] & \Filt \{\sst{2}, \sst{2\\1}\} = \TT_3 \\ \hline
      [0,\TT_4] & \Filt \{1,2\} = \TT_4   \\ \hline
      [\TT_2,\TT_3] & \Filt \{\sst{2\\1}\} = \add \sst{2\\1}
    \end{tabular}
  \end{table}
\end{example}

\begin{example}
  Let $Q$ be a 2-Kronecker quiver, that is, $Q:1\leftleftarrows2$, and consider its path algebra $kQ$.
  For any integer $m\geq 1$, we set
  \[
    P_{m}=\left\{ \begin{array}{ll}
      \tau^{-(m-1)/2}P(1) & ({\rm m \ is \ odd}) \\
      \tau^{-(m-2)/2}P(2) & ({\rm m \ is \ even})
    \end{array} \right. , \quad
   I_{m}=\left\{ \begin{array}{ll}
      \tau^{(m-1)/2}I(2) & ({\rm m \ is \ odd}) \\
      \tau^{(m-2)/2}I(1) & ({\rm m \ is \ even})
    \end{array} \right. ,
  \]
  where $P(i)$ (resp. $I(i)$) is the indecomposable projective (resp. injective) module
  corresponding to the vertex $i$, and $\tau$ is the Auslander-Reiten translation.  We denote by $R_{\lambda}$ the indecomposable regular brick corresponding to $\lambda\in \P^{1}(k)$.
  For each $X \subseteq \P^1(k)$, define $\TT(X):=\TTT(\{ R_\lambda\}_{\lambda \in X})$ for $X \neq \varnothing$ and $\TT(\varnothing):= \bigcap_{X \neq \varnothing}\TT(X) = \add \{ I_m\}_{m=1}^{\infty}$.
  Then the complete list of torsion classes in $\mod kQ$ is as follows:
  $\mod kQ = \Fac (P_1 \oplus P_2)$, $\Fac P_i = \Fac (P_i \oplus P_{i+1})$ for $i \geq 2$, $\add 1 = \add P_1$, $\Fac I_i$ for $i \geq 1$, and $\TT(X)$ for $X \subseteq\P^1(k)$.
  By using the description of $\Hasse (\tors kQ)$ given in \cite[Example 3.6]{DIRRT}, we can obtain the following list of ICE intervals and their hearts.
  \begin{table}[htp]
    \begin{tabular}{c|c}
      ICE intervals in $\tors kQ$ & ICE-closed subcategories \\ \hline \hline
      $[0,\TT]$ for each $\TT \in \tors kQ$ & $\TT$ \\ \hline
      $[\add P_1,\mod kQ]$ & $\add 2 = \add I_1$ \\ \hline
      $[\Fac I_i,\Fac I_{i+1}]$ for $i \geq 1$ & $\add I_{i+1}$ \\ \hline
      $[\Fac (P_{i+1}\oplus P_{i+2}),\Fac (P_i \oplus P_{i+1})]$ for $i \geq 1$
      & $\add P_i$  \\ \hline
      $[\TT(X),\TT(Y)]$ for $X \subsetneq Y \subseteq \P^1(k)$ & $\Filt \{ R_\lambda \}_{\lambda \in Y \setminus X}$ \\ \hline
      $[\TT,\TT]$ for each $\TT \in \tors kQ$ & $0$
    \end{tabular}
  \end{table}
\end{example}
See Section \ref{sec:ex} for further examples of ICE intervals and ICE-closed subcategories.

\section{Wide $\tau$-tilting modules}\label{sec:4}
In this section, we develop a way to classify ICE-closed subcategories of a given abelian category, especially $\mod\Lambda$ for an artin algebra $\Lambda$.
By Corollary \ref{cor:widetors}, we have a surjection
\[
\{ (\WW,\CC) \, | \, \text{ $\WW \in \wide\AA$ and $\CC \in \tors\WW$} \} \defl \ice\AA
\]
defined by $(\WW,\CC) \mapsto \CC$. However, this map is not an injection in general. For example, for any simple object $S$ in $\AA$, we have that $\Filt S$ is a torsion class in every wide subcategory which contains $S$.
One of the aim of this section is to modify the above map to obtain a bijection.

\subsection{Sincere torsion classes}
Let us begin by introducing \emph{sincere subcategories} of an abelian length category.
\begin{definition}
  Let $\AA$ be an abelian length category and $\CC$ a subcategory of $\AA$.
  \begin{enumerate}
    \item $\simp\AA$ denotes the set of isomorphism classes of simple objects in $\AA$.
    \item For $M \in \AA$, we denote by $\supp M$ the subset of $\simp\AA$ consisting of composition factors of $M$. In addition, we define $\supp\CC$ by $\supp\CC := \bigcup_{M \in \CC} \supp M$.
    \item $\CC$ is \emph{sincere} if $\supp \CC = \simp\AA$ holds, that is, every simple object in $\AA$ appears as a composition factor of some object in $\CC$.
    \item $\stors\AA$ denotes the set of sincere torsion classes in $\AA$.
  \end{enumerate}
  When we want to emphasize the abelian category $\AA$, we will write as $\supp_\AA M$ and $\supp_\AA \CC$.
\end{definition}
Then we can show that any ICE-closed subcategory can be realized as a \emph{sincere} torsion class in some wide subcategory:
\begin{proposition}\label{prop:icestors}
  Let $\CC$ be an ICE-closed subcategory of $\AA$. Then there is some wide subcategory $\WW$ of $\AA$ such that $\CC$ is a sincere torsion class in $\WW$.
\end{proposition}
\begin{proof}
  By Corollary \ref{cor:widetors}, there is a wide subcategory $\WW'$ such that $\CC$ is a torsion class in $\WW'$. Consider $\WW:= \Filt(\supp_{\WW'} \CC)$. Since $\WW$ is a Serre subcategory of $\WW'$, it is a wide subcategory of $\AA$, and clearly $\CC$ is a torsion class in $\WW$. Moreover, $\CC$ is a sincere subcategory of $\WW$ by construction.
  Thus $\CC$ is a sincere torsion class in $\WW$.
\end{proof}

Now we can state the first result of this section. For a subcategory $\CC$ of $\AA$, we denote by $\kker\CC$ the subcategory of $\AA$ consisting of kernels of morphisms in $\CC$.
\begin{theorem}\label{thm:classifyice}
  We have a bijection
  \[
  \begin{tikzcd}
    \{ (\WW,\CC) \, | \, \text{$\WW \in \wide\AA$ and $\CC \in \stors\WW$} \} \rar["\Theta", "\sim"'] & \ice\AA
  \end{tikzcd}
  \]
  given by $\Theta(\WW,\CC)= \CC$. Its inverse is given by $\CC \mapsto (\la\CC\ra_{\wide},\CC)$, where $\la\CC\ra_{\wide}$ is the smallest wide subcategory containing $\CC$. Moreover, we have
  $\la\CC\ra_{\wide} = \Filt(\kker\CC) = \WR(\FFF(\CC))$.
\end{theorem}
\begin{proof}
  This map $\Theta$ is surjective by Proposition \ref{prop:icestors}.
  Let $\CC$ be a sincere torsion class in a wide subcategory $\WW$. We claim that $\WW = \Filt(\kker\CC)$ holds.

  Since $\CC$ is contained in $\WW$ and $\WW$ is closed under kernels and extensions, $\Filt(\kker\CC) \subseteq \WW$ holds.
  Conversely, let $S$ be a simple object in $\WW$. Then since $\CC$ is a sincere subcategory of $\WW$, there is some $M \in \CC$ which contains $S$ as a composition factor in $\WW$ (note that we consider composition series in the abelian category $\WW$ here, not in $\AA$).
  This implies that there are two subobjects $ M_1 < M_2 \leq M$ of $M$ satisfying $M_1,M_2, M/M_1,M/M_2 \in \WW$ and $M_2/M_1 \iso S$.
  Then we have the following short exact sequence in $\WW$.
  \[
  \begin{tikzcd}
    0 \rar & S \rar & M/M_1 \rar & M/M_2 \rar & 0
  \end{tikzcd}
  \]
  Since $\CC$ is a torsion class in $\WW$, we have $M/M_1, M/M_2 \in \CC$. This implies $S \in \kker\CC$. Therefore, we have $\simp\WW \subseteq \kker\CC$, which implies $\WW = \Filt(\simp\WW) \subseteq \Filt(\kker\CC)$. Thus $\WW = \Filt(\kker\CC)$ holds.

  Suppose that $\Theta(\WW_1,\CC) = \CC = \Theta(\WW_2,\CC)$ holds. Then we have $\WW_1 = \Filt(\kker\CC) = \WW_2$ because $\CC$ is a sincere torsion class in both $\WW_1$ and $\WW_2$. Thus $\Theta$ is injective, so $\Theta$ is a bijection.

  We have shown that for an ICE-closed subcategory $\CC$, there is some $\WW \in \wide\AA$ satisfying $\CC \in \stors\WW$, and this $\WW$ has turned out to be equal to $\Filt(\kker\CC)$.
  Therefore, $\Filt(\kker\CC)$ is always a wide subcategory of $\AA$. Moreover, any wide subcategory $\WW'$ containing $\CC$ must contain $\Filt(\kker\CC)$ since $\WW'$ is closed under kernels and extensions. Hence $\la\CC\ra_{\wide} = \Filt(\kker\CC)$ holds.

  Finally, we will show $\WR(\FFF(\CC)) = \la\CC\ra_{\wide}$. Put $\FF:= \FFF(\CC)$ for simplicity.
  To show this, by the previous argument, it suffices to show that $\CC$ is a sincere torsion class in $\WR(\FF)$.
  Since Corollary \ref{cor:widetors} shows that $\CC$ is a torsion class in $\WR(\FF)$, it is enough to prove that $\CC$ is a sincere subcategory of $\WR(\FF)$.
  Let $S$ be a simple object in an abelian category $\WR(\FF)$. Suppose that $\AA(S,C) = 0$ for all $C \in \CC$. This implies that $\CC$ is contained in a torsion-free class $S^\perp$. Then we have $\FFF(\CC) \subseteq S^\perp$ by the minimality of $\FFF(\CC)$. Since $0 \neq S$ belongs to $\WR(\FF) \subseteq \FF = \FFF(\CC)$, this is a contradiction.
  Therefore, there is some non-zero map $\varphi \colon S \to C$ with $C \in \CC$. Since $C \in \CC \subseteq \WR(\FF)$ and $S$ is simple in $\WR(\FF)$, this map is an injection by the Schur's lemma in the abelian category $\WR(\FF)$. Thus we obtain the following short exact sequence.
  \[
  \begin{tikzcd}
    0 \rar & S \rar["\varphi"] & C \rar & \coker\varphi \rar & 0
  \end{tikzcd}
  \]
  Since $\WR(\FF)$ is a wide subcategory, $\coker\varphi$ belongs to $\WR(\FF)$, hence the above exact sequence is a short exact sequence in $\WR(\FF)$. This shows that $S \in \supp_{\WR(\FF)}(C) \subseteq \supp_{\WR(\FF)}(\CC)$. Therefore, $\CC$ is a sincere subcategory of $\WR(\FF)$.
\end{proof}
Thanks to this, the classification of ICE-closed subcategories is completely decomposed into the following two steps.
\begin{enumerate}
  \item First determine $\wide\AA$, the set of all wide subcategories in $\AA$.
  \item For each $\WW \in \wide\AA$, find all sincere torsion classes in $\WW$.
\end{enumerate}
Of course, one cannot expect a general method to achieve this, but if $\AA$ is the module category of artin algebras, we can make use of $\tau$-tilting theory developed in \cite{AIR}.

\subsection{Wide $\tau$-tilting modules and doubly functorially finite ICE-closed subcategories}
In the rest of this section, \emph{we denote by $\Lambda$ an artin $R$-algebra over a commutative artinian ring $R$}, that is, $\Lambda$ is an $R$-algebra which is finitely generated as an $R$-module.
We often omit the base ring $R$ and simply say that $\Lambda$ is an artin algebra. We denote by $\mod\Lambda$ (resp. $\proj\Lambda$) the category of finitely generated right (resp. projective) $\Lambda$-modules, and by $\tau$ the Auslander-Reiten translation. For $M \in \mod\Lambda$, we denote by $|M|$ the number of non-isomorphic indecomposable direct summands of $M$.
We refer the reader to \cite{ASS,ARS} for the basics of the representation theory of artin algebras.

For an abelian length category $\AA$, we denote by $\ftors\AA$, $\sftors\AA$, and $\fwide\AA$ the posets of functorially finite torsion classes, sincere functorially finite torsion classes, and functorially finite wide subcategories of $\AA$ respectively.
We simply write $\tors\Lambda := \tors(\mod\Lambda)$, and similarly $\torf\Lambda$, $\ftors\Lambda$, $\sftors\Lambda$, $\ice\Lambda$, $\wide\Lambda$ and $\fwide\Lambda$.

For an extension-closed subcategory $\CC$ of $\AA$, we define the following notions of $\Ext$-projectives.
\begin{itemize}
  \item $P \in \CC$ is \emph{$\Ext$-projective in $\CC$} if $\Ext_\AA^1(P,\CC) = 0$. We denote by $\PP(\CC)$ the subcategory of $\CC$ consisting of $\Ext$-projective objects in $\CC$.
  \item $\CC$ \emph{has enough $\Ext$-projectives} if, for every $C \in \CC$, there is a short exact sequence in $\AA$
  \[
  \begin{tikzcd}
    0 \rar & C' \rar & P \rar & C \rar & 0
  \end{tikzcd}
  \]
  satisfying $P \in \PP(\CC)$ and $C' \in \CC$.
  \item $\CC$ \emph{has an $\Ext$-progenerator $P$} if $\CC$ has enough $\Ext$-projectives and $\PP(\CC) = \add P$ holds.
  \item If $\CC$ has an $\Ext$-progenerator, then we denote by $P(\CC)$ the unique basic $\Ext$-progenerator, or equivalently, $P(\CC)$ is a direct sum of non-isomorphic indecomposable $\Ext$-projective objects in $\CC$.
\end{itemize}

Every ICE-closed subcategory can be recovered from its $\Ext$-progenerator if it exists, as the following lemma says. For $M \in \AA$, we denote by $\ccok M$ the subcategory of $\AA$ consisting of cokernels of morphisms in $\add M$.
\begin{lemma}\label{lem:icecok}
  Let $\CC$ be an ICE-closed subcategory of $\AA$. If $\CC$ has an $\Ext$-progenerator $M$, then we have $\CC = \ccok M$.
\end{lemma}
\begin{proof}
  Since $\CC$ is closed under cokernels and $M \in \CC$, we have $\ccok M \subseteq \CC$. Conversely, since $M$ is an $\Ext$-progenerator of $\CC$, it is easy to check that every object in $\CC$ is a cokernel of a morphism in $\add M$. Thus $\CC \subseteq \ccok M$ holds.
\end{proof}

Using $\Ext$-projectives, \cite{AIR} established a bijection between functorially finite (sincere) torsion classes in $\mod\Lambda$ and a certain class of modules defined below.
\begin{definition}
  Let $\Lambda$ be an artin algebra and $M \in \mod\Lambda$.
  \begin{enumerate}
    \item $M$ is \emph{$\tau$-rigid} if $\Hom_\Lambda(M,\tau M) = 0$.
    \item $M$ is \emph{$\tau$-tilting} if $M$ is $\tau$-rigid and $|M| = |\Lambda|$ holds. We denote by $\ttilt\Lambda$ the set of isomorphism classes of basic $\tau$-tilting $\Lambda$-modules.
    \item For $P \in \proj\Lambda$, a pair $(M,P)$ is a \emph{$\tau$-tilting pair} if $M$ is $\tau$-rigid, $\Hom_\Lambda(P,M) = 0$ and $|M| + |P| = |\Lambda|$. We denote by $\ttiltp \Lambda$ the set of isomorphism classes of basic $\tau$-tilting pairs.
    \item $M$ is \emph{support $\tau$-tilting} if there is some $P \in \proj\Lambda$ such that $(M,P)$ is a $\tau$-tilting pair. We denote by $\sttilt\Lambda$ the set of isomorphism classes of basic support $\tau$-tilting $\Lambda$-modules.
    \item $\Lambda$ is \emph{$\tau$-tilting finite} if $\sttilt\Lambda$ is a finite set.
  \end{enumerate}
\end{definition}
Now we collect some results on $\tau$-tilting theory which we need. Note that it was assumed in \cite{AIR} that $\Lambda$ is a finite-dimensional algebra over an algebraically closed field, but the same proof applies for an artin algebra.
\begin{theorem}\label{thm:ttt}
  Let $\Lambda$ be an artin algebra. Then the following hold.
  \begin{enumerate}
    \item \cite[Theorem]{Sma} $\TT \in \tors\Lambda$ is functorially finite if and only if so is $\TT^\perp \in \torf\Lambda$.
    \item \cite[Theorem 5.10]{AS2} $\Fac M$ is a functorially finite torsion class for a $\tau$-rigid module $M$.
    \item \cite[Lemma 2.6, Theorem 2.7, Corollary 2.8, Corollary 2.13]{AIR}
    A torsion class is functorially finite if and only if it has an $\Ext$-progenerator, and we have the following commutative diagram with all the horizontal maps bijective.
    \[
    \begin{tikzcd}
      \ttiltp \Lambda \rar["{(M,P)\mapsto M}"] & [2em]
      \sttilt \Lambda  \rar["\Fac", shift left] & \lar["P", shift left] \ftors\Lambda, \\
      & \ttilt \Lambda \rar["\Fac", shift left] \ar[lu, "{M \mapsto (M,0)}", hookrightarrow] \uar[symbol=\subseteq] & \lar["P", shift left] \sftors\Lambda. \uar[symbol=\subseteq]
    \end{tikzcd}
    \]
    \item \cite[Theorem 3.8]{DIJ} $\Lambda$ is $\tau$-tilting finite if and only if $\tors\Lambda = \ftors\Lambda$ holds if and only if $\tors\Lambda$ is a finite set.
  \end{enumerate}
\end{theorem}
Our aim is to extend this bijection to include ICE-closed subcategories. By Theorem \ref{thm:classifyice}, we have to deal with both a wide subcategory $\WW$ and a sincere torsion class $\CC$ in $\WW$. To study $\WW$ and $\CC$, we can make use of Theorem \ref{thm:ttt} if the following two conditions are satisfied:
\begin{itemize}
  \item $\WW$ is equivalent to $\mod \Gamma$ for some artin algebra $\Gamma$, and
  \item $\CC$ is functorially finite in $\WW$.
\end{itemize}
The first condition is known to be equivalent to the functorial finiteness of $\WW$.
\begin{lemma}\label{lem:fwide}
  A wide subcategory $\WW$ of $\mod\Lambda$ is functorially finite if and only if $\WW$ has an $\Ext$-progenerator if and only if there is an artin algebra $\Gamma$ such that $\WW$ is equivalent to $\mod\Gamma$.
\end{lemma}
This seems to be a folklore result, and see \cite[Proposition 4.12]{rigid} for the proof.
The above argument leads to the following finiteness condition on ICE-closed subcategories.
\begin{definition}\label{def:dfice}
  Let $\CC$ be an ICE-closed subcategory of $\mod\Lambda$.
  \begin{enumerate}
    \item $\CC$ is \emph{doubly functorially finite} if there is a functorially finite wide subcategory $\WW$ of $\mod\Lambda$ such that $\CC$ is a functorially finite torsion class in $\WW$.
    \item $\dfice\Lambda$ denotes the set of doubly functorially finite ICE-closed subcategories of $\mod\Lambda$.
  \end{enumerate}
\end{definition}
We have the following characterization of the doubly functorial finiteness.
\begin{proposition}\label{prop:dfice}
  For an ICE-closed subcategory $\CC$ of $\mod\Lambda$, the following are equivalent.
  \begin{enumerate}
    \item $\CC$ is doubly functorially finite.
    \item There is some functorially finite wide subcategory $\WW$ such that $\CC$ is a sincere functorially finite torsion class in $\WW$.
    \item $\CC$ and $\WR(\FFF(\CC))$ are both functorially finite.
  \end{enumerate}
  Consequently, if $\CC$ and $\FFF(\CC)$ are functorially finite, then $\CC$ is doubly functorially finite.
\end{proposition}
\begin{proof}
  (1) $\Rightarrow$ (2):
  Clearly $\CC$ is functorially finite in $\mod\Lambda$.
  Since $\CC$ is doubly functorially finite, there is $\WW' \in \fwide\Lambda$ satisfying $\CC \in \ftors\WW'$. By the proof of Proposition \ref{prop:icestors}, there is a Serre subcategory $\WW$ of $\WW'$ such that $\CC$ is a sincere torsion class in $\WW$.
  On the other hand, since $\WW'$ is equivalent to the module category of some artin algebra by Lemma \ref{lem:fwide}, so is its Serre subcategory $\WW$. Thus $\WW$ is also functorially finite in $\mod\Lambda$ by the same lemma.

  (2) $\Rightarrow$ (3):
  Clearly $\CC$ is functorially finite in $\mod\Lambda$.
  By Theorem \ref{thm:classifyice}, we have $\WW = \WR(\FFF(\CC))$ since $\CC$ is a sincere torsion class in $\WW$. Thus $\WR(\FFF(\CC))$ is functorially finite.

  (3) $\Rightarrow$ (1):
  Clear from Corollary \ref{cor:widetors}.

  To prove the remaining assertion, it suffices to show that $\WR(\FFF(\CC))$ is functorially finite if so is $\FFF(\CC)$. We refer the reader to \cite[Proposition 2.28]{asai} or \cite[Lemma 3.8]{MS} for the proof of this fact. Instead, this also follows from Lemma \ref{lem:2-3} below and Proposition \ref{prop:torfwide} (1).
\end{proof}
For example, by using this criterion, we will see later in Proposition \ref{prop:ttfice} that $\ice\Lambda = \dfice\Lambda$ holds for a $\tau$-tilting finite algebra $\Lambda$.
\begin{remark}
  There may exist a functorially finite wide subcategory $\WW$ of $\mod\Lambda$ such that $\FFF(\WW)$ is not functorially finite (see \cite[Example 4.13]{asai}). Such a wide subcategory is doubly functorially finite, thus the converse of the last assertion in Proposition \ref{prop:dfice} does not hold in general.
\end{remark}
We are in a position to introduce wide $\tau$-tilting modules.
\begin{definition}\label{def:wtt}
  Let $\Lambda$ be an artin algebra.
  \begin{enumerate}
    \item For $\WW \in \fwide\Lambda$ and $M \in \WW$, fix an equivalence $F \colon \WW \equi \mod\Gamma$ for an artin algebra $\Gamma$ (which exists by Lemma \ref{lem:fwide}). We say that $M$ is \emph{$\tau_\WW$-tilting} if $F(M)$ is a $\tau$-tilting $\Gamma$-module.
    \item A pair $(\WW,M)$ is a \emph{wide $\tau$-tilting pair} if $\WW \in \fwide\Lambda$ and $M \in \WW$ is $\tau_\WW$-tilting. This pair is called \emph{basic} if so is $M$, and we denote by $\wttiltp \Lambda$ the set of isomorphism classes of basic wide $\tau$-tilting pairs.
    \item $M$ is a \emph{wide $\tau$-tilting module} if there is some $\WW \in \fwide\Lambda$ such that $M$ is $\tau_\WW$-tilting. We denote by $\wttilt\Lambda$ the set of isomorphism classes of basic wide $\tau$-tilting modules.
  \end{enumerate}
\end{definition}

The relation between support $\tau$-tilting modules and wide $\tau$-tilting modules is summarized as follows: support $\tau$-tilting modules are nothing but ``Serre" $\tau$-tilting modules, as the following proposition shows.
\begin{proposition}\label{prop:serrettilt}
  We have the following commutative diagram,
  \begin{equation}\label{eq:kakan}
    \begin{tikzcd}
      \wttiltp \Lambda \rar["{(\WW,M)\mapsto M}"] & [2em] \wttilt\Lambda \\
      \ttiltp\Lambda \uar["\iota", hookrightarrow] \rar["{(M,P) \mapsto M}", "\sim"'] & \sttilt\Lambda \uar[symbol=\subseteq]
    \end{tikzcd}
  \end{equation}
  where $\iota$ is defined by $\iota(M,P) = (P^\perp,M)$. Moreover, $\iota$ is injective, and its image consists of $(\WW,M) \in \wttiltp\Lambda$ such that $\WW$ is a Serre subcategory of $\mod\Lambda$. Consequently, the following are equivalent for $M \in \mod\Lambda$.
  \begin{enumerate}
    \item $M$ is a support $\tau$-tilting module.
    \item There is some Serre subcategory $\WW$ of $\mod\Lambda$ such that $(\WW,M)$ is a wide $\tau$-tilting pair.
  \end{enumerate}
\end{proposition}
\begin{proof}
  Note that the bottom map $\ttiltp\Lambda \to \sttilt\Lambda$ in (\ref{eq:kakan}) is a bijection by Theorem \ref{thm:ttt} (3).
  For an idempotent $e$ of $\Lambda$, we identify $\mod (\Lambda / \la e \ra)$ with the subcategory of $\mod\Lambda$ by the natural functor $\mod (\Lambda/\la e \ra) \hookrightarrow \mod\Lambda$.
  Then a subcategory $\WW$ of $\mod\Lambda$ is a Serre subcategory if and only if it arises in this way.

  First, we show that $\iota$ is well-defined.
  Let $(M,P)$ be a $\tau$-tilting pair. Then $P^\perp$ is a Serre subcategory of $\mod\Lambda$ since $P$ is projective, thus it is a wide subcategory. Moreover, there is an idempotent $e$ of $\Lambda$ satisfying $P^\perp = \mod (\Lambda/\la e \ra)$.
  Then $M$ is a $\tau$-tilting $(\Lambda/\la e \ra)$-module by \cite[Proposition 2.3]{AIR}, thus $(P^\perp,M)$ is a wide $\tau$-tilting pair. Thus $\iota$ gives a map $\ttiltp\Lambda \to \wttiltp\Lambda$.
  Since the bottom map of (\ref{eq:kakan}) is surjective, clearly every support $\tau$-tilting module is wide $\tau$-tilting, and we obtain the commutative diagram (\ref{eq:kakan}). Moreover, since the bottom map is bijective, $\iota$ is an injection.

  Next, we discuss the image of $\iota$. By the above argument, if $(\WW,M) \in \im \iota$, then $\WW$ is a Serre subcategory.
  Conversely, suppose that $(\WW,M)$ is a wide $\tau$-tilting pair and $\WW$ is a Serre subcategory. Then there is an idempotent $e$ of $\Lambda$ satisfying $\WW = \mod(\Lambda/\la e \ra)$, and it is easy to check that $(M,e\Lambda)$ is a $\tau$-tilting pair satisfying $\iota(M,e\Lambda) = (\WW,M)$.

  Finally, the equivalence of (1) and (2) is clear from the above discussion.
\end{proof}
We will see later in Proposition \ref{prop:wtiltpbij} that the map $\wttiltp\Lambda \to \wttilt\Lambda$ in (\ref{eq:kakan}) is a bijection.
Now we are ready to establish a bijection between wide $\tau$-tilting modules and doubly functorially finite ICE-closed subcategories.
\begin{theorem}\label{thm:wttiltbij}
  Let $\Lambda$ be an artin algebra. Then every doubly functorially finite ICE-closed subcategories has an $\Ext$-progenerator, and we have mutually inverse bijections
  \[
  \begin{tikzcd}[row sep = 0]
    \wttilt\Lambda \rar[shift left, "\ccok"] & \dfice\Lambda \lar[shift left, "P(-)"] \\
  \end{tikzcd}
  \]
  between wide $\tau$-tilting modules and doubly functorially finite ICE-closed subcategories of $\mod\Lambda$.
\end{theorem}
\begin{proof}
  Let $M$ be a wide $\tau$-tilting module. Then we have a functorially finite wide subcategory $\WW$ such that $M$ is $\tau_\WW$-tilting.
  Since $\WW$ is equivalent to the module category of an artin algebra and $M$ is $\tau$-tilting in it, $\Fac_\WW M$ is a sincere functorially finite torsion class in $\WW$ and has an $\Ext$-progenerator $M$ by Theorem \ref{thm:ttt} (3) (here $\Fac_\WW$ means we consider $\Fac$ inside the abelian category $\WW$).
  Then $\Fac_\WW M$ is ICE-closed in $\mod\Lambda$ by Lemma \ref{lem:torsice}. Moreover, it is doubly functorially finite by definition, and $\Fac_\WW M = \ccok M$ holds by Lemma \ref{lem:icecok}.
  Therefore, we obtain a map $\ccok \colon \wttilt\Lambda \to \dfice\Lambda$.

  Conversely, let $\CC$ be a doubly functorially finite ICE-closed subcategory. Then by Proposition \ref{prop:dfice}, there is some functorially finite wide subcategory $\WW$ such that $\CC$ is a sincere torsion class in $\WW$.
  Then we obtain an equivalence $F \colon \WW \equi \mod\Gamma$ for some artin algebra $\Gamma$ by Lemma \ref{lem:fwide}, and $F(\CC)$ is a sincere functorially finite torsion class in $\mod\Gamma$. Therefore, $F(\CC)$ has an $\Ext$-progenerator, which is a $\tau$-tilting $\Gamma$-module by Theorem \ref{thm:ttt} (3).
  Thus $\CC$ has an $\Ext$-progenerator $P(\CC)$ which is $\tau_\WW$-tilting. Hence $P(\CC)$ is wide $\tau$-tilting, and we obtain a map $P(-) \colon \dfice\Lambda \to \wttilt\Lambda$.

  It is straightforward to check that these maps are inverse to each other, so we omit it.
\end{proof}

We now take a closer look at wide $\tau$-tilting pairs and wide $\tau$-tilting modules. Each wide $\tau$-tilting module uniquely determines its defining wide subcategory, as the following claims.
\begin{proposition}\label{prop:wtiltpbij}
  Let $\Lambda$ be an artin algebra. Then the map $\wttiltp\Lambda \to \wttilt\Lambda$ given by $(\WW,M) \mapsto M$ is a bijection, and its inverse is given by $M \mapsto (\la M \ra_{\wide},M)$, and we have $\la M \ra_{\wide} = \WR(\FFF(\ccok M)) = \Filt(\kker (\ccok M) )$.
\end{proposition}
\begin{proof}
  This map is surjective by definition. On the other hand, for $(\WW,M) \in \wttiltp\Lambda$, the proof of Theorem \ref{thm:wttiltbij} implies that $\ccok M$ is a sincere torsion class in $\WW$.
  Then Theorem \ref{thm:classifyice} implies $\WW = \WR(\FFF(\ccok M)) = \la \ccok M \ra_{\wide}$, hence this map is injective. Moreover, clearly we have $\la \ccok M \ra_{\wide} = \la M \ra_{\wide}$.
\end{proof}

Summarizing these results, we immediately obtain the following.
\begin{corollary}\label{cor:compati}
  Let $\Lambda$ be an artin algebra. Then we have the following commutative diagram, and all the horizontal maps are bijective.
  \[
  \begin{tikzcd}
    \wttiltp\Lambda \rar["{(\WW,M)\mapsto M}"] & [2em] \wttilt\Lambda \rar[shift left, "\ccok"] & \dfice\Lambda \lar[shift left, "P(-)"] \\
    \ttiltp\Lambda \uar["{(M,P) \mapsto (P^\perp,M)}", hookrightarrow] \rar["{(M,P) \mapsto M}"] & \sttilt\Lambda \rar[shift left, "\Fac"] \uar[symbol=\subseteq] & \ftors\Lambda \uar[symbol=\subseteq] \lar[shift left, "P(-)"]
  \end{tikzcd}
  \]
\end{corollary}
\begin{proof}
  This follows from Propositions \ref{prop:serrettilt} and \ref{prop:wtiltpbij} and Theorems \ref{thm:ttt} (3) and \ref{thm:wttiltbij}, except the fact that $\Fac M = \ccok M$ for $M \in \sttilt\Lambda$.
  This follows from Lemma \ref{lem:icecok} since $\Fac M$ is a torsion class with an $\Ext$-progenerator $M$ for $M \in \sttilt\Lambda$ by Theorem \ref{thm:ttt} (3).
\end{proof}

\subsection{Wide $\tau$-tilting modules from ICE intervals}
In the rest of this section, we study how to obtain wide $\tau$-tilting modules from support $\tau$-tilting modules.

When studying wide $\tau$-tilting modules, working on $\tors\Lambda$ is more natural than $\torf\Lambda$ since we can use $\tau$-tilting modules to describe $\Ext$-progenerators of hearts, see Corollary \ref{cor:fromttint}. Therefore, \emph{from now on, we mainly work on intervals of $\tors\Lambda$}.
Recall that the heart of an interval $[\UU,\TT]$ in $\tors\Lambda$ is defined by $\HH_{[\UU,\TT]}:= \TT \cap \UU^\perp$.
By Remark \ref{rem:heart}, considering hearts of intervals in $\tors\Lambda$ is essentially the same as in $\torf\Lambda$ via the anti-isomorphisms of posets $(-)^\perp \colon \tors\Lambda \rightleftarrows \torf\Lambda \colon {}^\perp (-)$.

First, we recall the \emph{torsion radical} and the \emph{torsion-free coradical} associated to torsion pairs. A pair $(\TT,\FF)$ of subcategories of $\mod\Lambda$ is a \emph{torsion pair} if $\TT$ is a torsion class, $\FF$ is a torsion-free class, $\TT = {}^\perp \FF$ and $\FF = \TT^\perp$.
\begin{definition}
  Let $(\TT,\FF)$ be a torsion pair in $\mod\Lambda$. Since $\mod\Lambda = \TT * \FF$ holds by Lemma \ref{prop:torsbasic}, for $M \in \mod\Lambda$, we have a short exact sequence
  \[
  \begin{tikzcd}
    0 \rar & tM \rar & M \rar & fM \rar & 0
  \end{tikzcd}
  \]
  with $tM \in \TT$ and $fM \in \FF$. It is easily checked that this sequence is uniquely determined by $M$ up to isomorphism, and the assignments $M \mapsto tM$ and $M \mapsto fM$ define functors $t \colon \mod\Lambda \to \TT$ and $f\colon \mod\Lambda \to \FF$.
  We call $t$ the \emph{torsion radical of $\TT$} and $f$ the \emph{torsion-free coradical of $\FF$}.
\end{definition}

Now we describe the $\Ext$-progenerator of the heart of an interval $[\UU,\TT]$ of $\tors\Lambda$ if the $\Ext$-progenerator of $\TT$ is given.
\begin{lemma}\label{lem:heartprogen}
  Let $[\UU,\TT]$ be an interval in $\tors\Lambda$, and suppose that $\TT$ has an $\Ext$-progenerator $T$. Let $g \colon \mod\Lambda \to \UU^\perp$ be the torsion-free coradical of $\UU^\perp$.
  Then $g$ induces a functor $\TT \to \HH_{[\UU,\TT]}$, and $\HH_{[\UU,\TT]}$ has an $\Ext$-progenerator $gT$.
\end{lemma}
\begin{proof}
  This is a special case of \cite[Proposition 5 (c)]{kalck}, but the proof was omitted there. Thus we give a proof for the convenience of the reader.

  Put $\HH = \HH_{[\UU,\TT]}$ for simplicity. For any $M \in \TT$, we have a surjection $M \defl gM$, thus $gM \in \TT$ holds. Therefore $g \colon \mod\Lambda \to \UU^\perp$ induces a functor $\TT \to \TT \cap \UU^\perp = \HH$.

  We first show that $gT$ is $\Ext$-projective in $\HH$.
  Denote by $u \colon \mod\Lambda \to \UU$ the torsion radical of $\UU$.
  Take any short exact sequence in $\mod\Lambda$
  \[
  \begin{tikzcd}
    0 \rar & H_{1} \rar & H_{2} \rar["p"] & gT \rar & 0
  \end{tikzcd}
  \]
  with $H_1 \in \HH$. Then $H_2 \in \HH$ since $\HH$ is extension-closed. Consider the following diagram.
  \[
  \begin{tikzcd}
    0 \rar & uT \rar["\iota"] & T \rar["\pi"] \ar[d, dashed, "\varphi"'] & gT \rar \dar[equal]   & 0 \\
    0 \rar & H_{1} \rar & H_{2} \rar["p"'] & gT \rar & 0
  \end{tikzcd}
  \]
  Since $T$ is $\Ext$-projective in $\TT$, we have $\Ext_\Lambda^1(T,H_1) = 0$, thus obtain a morphism $\varphi$ satisfying $p\varphi = \pi$. Then $\varphi\iota = 0$ since $uT \in \UU$ and $H_2 \in \HH \subseteq \UU^\perp$.
  Thus $\varphi$ factors through $\pi$. Now it is easy to see that $p$ is a retraction, hence $gT$ is $\Ext$-projective in $\HH$.

  Let $H$ be an object in $\HH$. Since $\TT$ has an $\Ext$-progenerator $T$, there is a short exact sequence
  \[
  \begin{tikzcd}[row sep=0]
    0 \rar & L \rar["\al"] & T_0 \rar["\be"] & H \rar & 0
  \end{tikzcd}
  \]
  satisfying $T_0 \in \add T$ and $L \in \TT$. We have the following short exact sequence by $\mod\Lambda = \UU * \UU^\perp$.
  \[
  \begin{tikzcd}[row sep=0]
    0 \rar & uT_0 \rar["\ga"] & T_0 \rar & gT_0 \rar & 0
  \end{tikzcd}
  \]
  We clearly have $gT_0 \in \add gT$ since $g$ is an additive functor.
  Now $\be\ga = 0$ by $uT_0 \in \UU$ and $H \in \HH \subseteq \UU^\perp$. Thus $\ga$ factors through $\al$, and we obtain the following exact commutative diagram.
  \[
  \begin{tikzcd}
    & 0 \dar & 0 \dar \\
    & uT_0 \rar[equal] \dar ["\ov{\ga}"] & uT_0 \dar["\ga"] \\
    0 \rar & L \rar["\al"]\dar & T_0 \rar["\be"] \dar & H \dar[equal]\rar & 0 \\
    0 \rar & \coker \ov{\ga} \dar \rar & gT_0 \dar \rar & H \rar & 0 \\
    & 0 & 0
  \end{tikzcd}
  \]
  Then $\coker\ov{\ga}$ is in $\HH = \TT \cap \UU^\perp$ because it is a quotient of $L  \in\TT$ and a subobject of $gT_0 \in \UU^\perp$. Thus the bottom exact sequence implies that $\HH$ has an $\Ext$-progenerator $gT$.
\end{proof}

Using this, we obtain the following useful criterion on the functorial finiteness.
\begin{lemma}\label{lem:2-3}
  Let $[\UU,\TT]$ be an interval in $\tors\Lambda$ with its heart $\HH:= \HH_{[\UU,\TT]}$. If two of $\UU$, $\TT$ and $\HH$ are functorially finite, then so is the third. In addition, assume that $[\UU,\TT]$ is a wide interval. Then $\UU$, $\TT$ and $\HH$ are all functorially finite if either $\UU$ or $\TT$ is functorially finite.
\end{lemma}
\begin{proof}
  Put $\TT^\perp = \FF$ and $\UU^\perp = \GG$ so that $(\TT,\FF)$ and $(\UU,\GG)$ are torsion pairs in $\mod\Lambda$.

  ($\UU,\TT \Rightarrow \HH$)
  Suppose that $\UU$ and $\TT$ are functorially finite. Then $\GG$ is also functorially finite by Theorem \ref{thm:ttt} (1). We claim that $\HH = \TT \cap \GG$ is functorially finite.

  Let $X \in \mod\Lambda$. Take a right $\GG$-approximation $G \to X$ of $X$. Then take a right $\TT$-approximation $T_G \hookrightarrow G$ of $G$, which is injective since $\TT$ is a torsion class. Then $T_G$ belongs to $\GG$ since $\GG$ is closed under submodules. Thus $T_G \in \TT \cap \GG = \HH$ holds.
  Now it is straightforward to see that the composition $T_G \hookrightarrow G \to X$ is a right $\HH$-approximation of $X$.

  Similarly, take a left $\TT$-approximation $X \to T^X$ of $X$, and then take a left $\GG$-approximation $T^X \defl G'$ of $T^X$, which is a surjection since $\GG$ is a torsion-free class. Then $G' \in \HH$ holds, and the composition $X \to T^X \defl G'$ is a left $\HH$-approximation of $X$. Thus $\HH$ is functorially finite.

  ($\UU,\HH \Rightarrow \TT$)
  Suppose that $\UU$ and $\HH$ are functorially finite. We have $\TT = \UU * \HH$ by the dual of Lemma \ref{lem:interval}. Now the assertion follows from the fact that an ``extension" of two functorially finite subcategories in $\mod\Lambda$ is functorially finite, see \cite[Theorem 1.1]{GT} or \cite{SS}.

  ($\TT,\HH \Rightarrow \UU$)
  Suppose that $\TT$ and $\HH$ are functorially finite. Then $\FF$ is functorially finite by Theorem \ref{thm:ttt} (1), and we have $\GG = \HH * \FF$ by Lemma \ref{lem:interval}. Thus the same argument as in ($\UU,\HH \Rightarrow \TT$) implies that $\GG$ is functorially finite. Thus Theorem \ref{thm:ttt} (1) shows that $\UU$ is functorially finite.

  Finally, assume that $\HH$ is a wide subcategory. If $\TT$ is functorially finite, then $\TT$ has an $\Ext$-progenerator by Theorem \ref{thm:ttt} (3), thus so does $\HH = \HH_{[\UU,\TT]}$ by Lemma \ref{lem:heartprogen}.
  Thus $\HH$ is functorially finite by Lemma \ref{lem:fwide}. This implies that $\UU$ is functorially finite by ($\TT,\HH \Rightarrow \UU$).
  Similarly, suppose that $\UU$ is functorially finite. Then $\UU^\perp$ is also functorially finite by Theorem \ref{thm:ttt} (1). Let $D\colon \mod\Lambda \leftrightarrow \mod \Lambda^{\op}$ denote the standard Matlis duality. Then $[D(\TT^\perp), D(\UU^\perp)]$ is a wide interval in $\tors\Lambda^{\op}$ with its heart $D\HH$, and $D(\UU^\perp)$ is functorially finite. Thus the above discussion shows that $D(\TT^\perp)$ and $D\HH$ are functorially finite in $\mod\Lambda^{\op}$. Therefore, $\TT^\perp$ and $\HH$ are functorially finite in $\mod\Lambda$, and $\TT$ is also functorially finite by Theorem \ref{thm:ttt} (1).
\end{proof}

As an immediate corollary, we obtain a description of an $\Ext$-progenerator of $\HH_{[\Fac U,\Fac T]}$ for $U,T \in \sttilt\Lambda$. Recall that the \emph{trace} $\tr_U(M)$ for $M \in \mod\Lambda$ is a submodule of $M$ defined by
\[
  \tr_U(M) := \sum \{ \im \varphi  \, | \, \varphi \colon U \to M \}.
\]
\begin{corollary}\label{cor:fromttint}
  Let $[\Fac U,\Fac T]$ be an interval in $\tors\Lambda$ with $U, T\in\sttilt\Lambda$. Then $\HH_{[\Fac U,\Fac T]}$ has an $\Ext$-progenerator $T/\tr_U(T)$.
  In particular, if $[\Fac U, \Fac T]$ is ICE interval, then $T/\tr_U(T)$ is a wide $\tau$-tilting module satisfying $\ccok (T/\tr_U(T)) = \HH_{[\Fac U,\Fac T]}$.
\end{corollary}
\begin{proof}
  Put $\CC:= \HH_{[\Fac U,\Fac T]}$ for simplicity.
  Since $\Fac U$ is a torsion class, $\tr_U(-)$ is the torsion radical of $\Fac U$, thus $M \mapsto M/\tr_U(M)$ is the torsion-free coradical of $U^\perp = (\Fac U)^\perp$.
  Thus $\CC$ has an $\Ext$-progenerator $T/\tr_U(T)$ by Lemma \ref{lem:heartprogen}.

  Suppose in addition that $\CC$ is ICE-closed. It suffices to show that $\CC$ is doubly functorially finite by Theorem \ref{thm:wttiltbij}.
  By the dual of Theorem \ref{thm:iceintchar}, there is a torsion class $\TT'$ in $\mod\Lambda$ such that $[\Fac U, \TT']$ is a wide interval and $\CC$ is a torsion class in a wide subcategory $\WW:=\HH_{[\Fac U,\TT']}$.
  Then Lemma \ref{lem:2-3} implies that $\WW$ is a functorially finite wide subcategory. Thus it is equivalent to $\mod\Gamma$ for some artin algebra $\Gamma$ by Lemma \ref{lem:fwide}. Since $\CC$ has an $\Ext$-progenerator, Theorem \ref{thm:ttt} (3) (applied to $\mod\Gamma$) implies that $\CC$ is functorially finite in $\WW$.
\end{proof}
As for the $\tau$-tilting finiteness, we prove the following observation. We remark that the equivalence of (1) and (2) below was essentially shown in \cite[Theorem 5.5]{enomono}.
\begin{proposition}\label{prop:ttfice}
  Let $\Lambda$ be an artin algebra. Then the following are equivalent.
  \begin{enumerate}
    \item $\Lambda$ is $\tau$-tilting finite.
    \item $\ice\Lambda$ is a finite set.
    \item Every ICE-closed subcategory of $\mod\Lambda$ is functorially finite.
    \item Every ICE-closed subcategory of $\mod\Lambda$ is doubly functorially finite.
    \item Every ICE-closed subcategory of $\mod\Lambda$ has an $\Ext$-progenerator.
  \end{enumerate}
\end{proposition}
\begin{proof}
  Clearly each of (2)-(5) implies (1) by Theorem \ref{thm:ttt} (4) since every torsion class is an ICE-closed subcategory. We will prove (1) $\Rightarrow$ (2) and (1) $\Rightarrow$ (3) $\Rightarrow$ (4) $\Rightarrow$ (5).

  (1) $\Rightarrow$ (2):
  Suppose that $\Lambda$ is $\tau$-tilting finite. Then $\tors\Lambda$ is a finite set by Theorem \ref{thm:ttt} (4). Therefore, there are only finitely many intervals in $\tors\Lambda$, which implies that $\ice\Lambda$ is a finite set by Proposition \ref{prop:arashi}.

  (1) $\Rightarrow$ (3):
  Let $\CC$ be an ICE-closed subcategory of $\mod\Lambda$. Then $\CC = \HH_{[\UU,\TT]}$ holds for some interval $[\UU,\TT]$ in $\tors\Lambda$ by Proposition \ref{prop:arashi}.
  Since $\Lambda$ is $\tau$-tilting finite, both $\UU$ and $\TT$ are functorially finite by Theorem \ref{thm:ttt}. Thus so is $\CC$ by Lemma \ref{lem:2-3}.

  (3) $\Rightarrow$ (4):
  Let $\CC$ be an ICE-closed subcategory of $\mod\Lambda$. Then $\CC$ is a torsion class in some wide subcategory $\WW$ by Corollary \ref{cor:widetors}. By (3), both of $\CC$ and $\WW$ are functorially finite (since $\WW$ is also ICE-closed), thus $\CC$ is doubly functorially finite.

  (4) $\Rightarrow$ (5):
  We have $\ice\Lambda = \dfice\Lambda$ by (4). By Theorem \ref{thm:wttiltbij}, this implies that every ICE-closed subcategory has an $\Ext$-progenerator.
\end{proof}

As a corollary, we obtain an algorithm to calculate all wide $\tau$-tilting modules for a $\tau$-tilting finite algebra.
We define the poset structure on $\sttilt\Lambda$ by $U \leq T$ if $\Fac U \subseteq \Fac T$.
\begin{definition-proposition}\label{def-prop:u+}
  Let $U$ be a support $\tau$-tilting module. Then the following join exists in the poset $\sttilt\Lambda$
  \[
  U^+ := U \vee \bigvee \{ U' \in \sttilt\Lambda \, | \, \text{there is an arrow $U \ot U'$ in $\Hasse(\sttilt\Lambda)$} \},
  \]
  and it satisfies the following:
  \[
  \Fac U^+ = \UU^+:= \UU \vee \bigvee \{ \UU' \in \tors\Lambda \, | \, \text{there is an arrow $\Fac U \ot \UU'$ in $\Hasse(\tors\Lambda)$} \}
  \]
\end{definition-proposition}
\begin{proof}
  By \cite[Theorem 3.1]{DIJ}, $\Hasse(\ftors\Lambda)$, which is isomorphic to $\Hasse(\sttilt\Lambda)$, is as a full subquiver of $\Hasse(\tors\Lambda)$, and if there is an arrow $\Fac U \ot \UU'$ in $\Hasse(\tors\Lambda)$, then $\UU'$ is functorially finite.
  Therefore, to show the claim, it suffices to prove that $\UU^+$ is functorially finite.
  By the dual of Proposition \ref{prop:torfwide}, we have that $[\Fac U, \UU^+]$ is a wide interval.
  Then $\UU^+$ is functorially finite by Lemma \ref{lem:2-3}.
\end{proof}
By using this $U^+$  and Corollaries \ref{cor:ice} and \ref{cor:fromttint}, we immediately obtain the following method to calculate $\wttilt\Lambda$ for the $\tau$-tilting finite case.
\begin{corollary}\label{cor:wttiltfromint}
  Let $\Lambda$ be a $\tau$-tilting finite artin algebra. Then all the wide $\tau$-tilting modules and ICE-closed subcategories of $\mod\Lambda$ can be obtained as follows.
  \begin{enumerate}
    \item Choose $U \in \sttilt\Lambda$.
    \item Compute $U^+ := U \vee \bigvee \{ U' \, | \, \text{there is an arrow $U \ot U'$ in $\Hasse(\sttilt\Lambda)$} \}$.
    \item For each $T \in \sttilt\Lambda$ with $U \leq T \leq U^+$, we have that $T/\tr_U(T)$ is a wide $\tau$-tilting module and $\ccok(T/\tr_U(T))$ is an ICE-closed subcategory of $\mod\Lambda$.
  \end{enumerate}
\end{corollary}
We refer the reader to Section \ref{sec:ex} for examples.

\subsection{Bijections between wide $\tau$-tilting modules and sincere intervals}
The method we gave in Corollary \ref{cor:wttiltfromint} is suitable for the actual computation, but the same wide $\tau$-tilting modules (thus the same ICE-closed subcategories) may appear several times.
Thus, from the theoretical viewpoint, it is natural to modify it to obtain some bijection, which is the aim of this subsection.
\emph{In what follows, for an interval $[U,T]$ in $\sttilt\Lambda$, we write $\HH_{[U,T]} := \HH_{[\Fac U, \Fac T]} = \Fac T \cap U^\perp$.}

By Proposition \ref{prop:wtiltpbij}, we only have to study $\wttiltp\Lambda$, that is, study functorially finite wide subcategories $\WW$ and $\tau_\WW$-tilting modules for each $\WW$, or equivalently, such $\WW$ and functorially finite sincere torsion class in $\WW$.
For the $\tau$-tilting finite case, we can use the following classification of wide subcategories.
\begin{proposition}\label{prop:wideclassify}
  Let $U \in \sttilt\Lambda$.
  Then $[\Fac U,\Fac U^+]$ is a wide interval, and we have a map
  \[
  \begin{tikzcd}[row sep=0]
    \sttilt\Lambda \rar & \wide\Lambda\\
    U \rar[mapsto] & \WW_U:=\HH_{[\Fac U,\Fac U^+]}.
  \end{tikzcd}
  \]
  Moreover, if $\Lambda$ is $\tau$-tilting finite, then this map is bijective.
\end{proposition}
\begin{proof}
  The first statement follows from the dual of Proposition \ref{prop:torfwide}.
  Suppose that $\Lambda$ is $\tau$-tilting finite. Then \cite[Corollary 3.11]{MS} shows that $\FF \mapsto \WR(\FF)$ is a bijection between $\torf\Lambda$ and $\wide\Lambda$ (see \cite[Corollary 5.4]{enomono} for an alternative proof). Thus the assertion follows via the anti-isomorphisms $(-)^\perp \colon \tors\Lambda \rightleftarrows \torf\Lambda \colon {}^\perp(-)$.
\end{proof}

To deal with $\tau_\WW$-tilting modules and sincere torsion classes in $\WW$, we introduce the following special intervals in $\sttilt\Lambda$.

\begin{definition}\label{def:sincere-int}
  Let $[U,T]$ be an interval in $\sttilt\Lambda$.
  We say that $[U,T]$ is a \emph{sincere interval} if $U \leq T \leq U^+$ holds and $\HH_{[U, T]}$ is a sincere subcategory of the wide subcategory $\WW_U := \HH_{[U, U^+]}$.
  We denote by $\sint\Lambda$ the set of sincere intervals in $\sttilt\Lambda$.
\end{definition}
These intervals are in bijection with wide $\tau$-tilting modules and ICE-closed subcategories if $\Lambda$ is $\tau$-tilting finite.
\begin{proposition}\label{prop:sint}
  Let $\Lambda$ be a $\tau$-tilting finite artin algebra. Then we have the following commutative diagram consisting of bijections,
  \[
  \begin{tikzcd}[column sep = large]
    \sint\Lambda \rar["\Phi"] \ar[rd,"\HH_{(-)}"'] & \wttiltp\Lambda \rar["{(\WW,M) \mapsto M}"] & \wttilt\Lambda \ar[ld, "\ccok"] \\
    & \ice\Lambda = \dfice\Lambda
  \end{tikzcd}
  \]
  where $\Phi[U,T] := (\WW_U, T/\tr_U(T))$.
\end{proposition}
\begin{proof}
  Since $\Lambda$ is $\tau$-tilting finite, $\dfice\Lambda = \ice\Lambda$ holds by Proposition \ref{prop:ttfice}.
  We shall show that $\Phi$ and $\HH_{(-)}$ are well-defined. Let $[U,T] \in \sint\Lambda$. Since $[\Fac U, \Fac U^+]$ is a wide interval with its heart $\WW_U$ and $\Fac U \subseteq \Fac T \subseteq \Fac U^+$, the interval $[\Fac U, \Fac T]$ is an ICE interval by Corollary \ref{cor:ice}, hence $\HH_{(-)}$ is well-defined.
  Moreover, $\HH_{[U,T]}$ is a sincere torsion class in $\WW_U$ by the dual of Proposition \ref{prop:wideint} (2) and the definition of sincere intervals, and has an $\Ext$-progenerator $T/\tr_U(T)$ by Corollary \ref{cor:fromttint}.
  Therefore, $T/\tr_U(T)$ is $\tau_{\WW_U}$-tilting by Theorem \ref{thm:ttt} (3). Hence $(\WW_U, T/\tr_U(T))$ is a wide $\tau$-tilting pair.

  It is clear that the above diagram commutes, and the maps between $\wttiltp\Lambda$, $\wttilt\Lambda$ and $\dfice\Lambda$ are bijective by Corollary \ref{cor:compati}.
  Therefore, it suffices to show that $\Phi$ is injective and $\HH_{(-)}$ is surjective.

  To show that $\Phi$ is injective, assume $\Phi[U_1,T_1] = \Phi[U_2,T_2]$ for $[U_1,T_1],[U_2,T_2] \in \sint\Lambda$.
  In particular, $\WW_{U_1} = \WW_{U_2}$ holds, which implies $U_1 = U_2$ by Proposition \ref{prop:wideclassify}. Put $U:= U_1 = U_2$.
  Since $\Lambda$ is $\tau$-tilting finite, $\Fac \colon \sttilt\Lambda \to \tors\Lambda$ is a bijection. Thus, the dual of Proposition \ref{prop:wideint} (2) implies the following bijection:
  \begin{equation}\label{eq}
  \begin{tikzcd}[row sep=0]
    [U, U^+] \rar["\sim"] & \tors \WW_U, \\
    T \rar[mapsto] & \HH_{[U,T]}.
  \end{tikzcd}
  \end{equation}
  Since $\Phi[U,T_1] = \Phi[U,T_2]$, we have $\HH_{[U,T_1]} = \HH_{[U,T_2]}$, which implies $T_1 = T_2$ by the bijection (\ref{eq}).

  Finally, we will show that $\HH_{(-)}$ is surjective. Let $\CC$ be an ICE-closed subcategory of $\mod\Lambda$. Theorem \ref{thm:classifyice} implies that there is a wide subcategory $\WW$ of $\mod\Lambda$ such that $\CC$ is a sincere torsion class in $\WW$. Since $\Lambda$ is $\tau$-tilting finite, there is some $U \in \sttilt\Lambda$ satisfying $\WW = \WW_U$ by Proposition \ref{prop:wideclassify}. By the bijection (\ref{eq}), there is some $T \in \sttilt\Lambda$ satisfying $U \leq T \leq U^+$ and $\HH_{[U,T]} = \CC$.
  Since $\CC$ is a sincere subcategory of $\WW_U$, we have $[U,T] \in \sint\Lambda$.
\end{proof}

In the rest of this subsection, we give two methods to obtain $\sint\Lambda$ for a $\tau$-tilting finite algebra $\Lambda$.
The first one (Corollary \ref{cor:howtowttilt}) is a module-theoretic method where we count certain indecomposable modules, and the second one (Corollary \ref{cor:posetstr}) is a combinatorial method where we only uses the poset structure of $\sttilt\Lambda$.
Once we obtain $\sint\Lambda$ in either method, we obtain all wide $\tau$-tilting modules and ICE-closed subcategories without duplication using Proposition \ref{prop:sint}.

The first method relies on the following observation on the number of indecomposable $\Ext$-projectives.
\begin{lemma}\label{lem:sincere-number}
  Let $[U,T]$ be an interval in $\sttilt\Lambda$ satisfying $U \leq T \leq U^+$. Then $[U,T]$ is a sincere interval if and only if $|P(\WW_U)| = |P(\HH_{[U,T]})|$ holds.
\end{lemma}
\begin{proof}
  Corollary \ref{cor:fromttint} shows that $\HH_{[U,T]}$ has an $\Ext$-progenerator $M:= P(\HH_{[U,T]}) = T/\tr_U(T)$.
  On the other hand, $\WW_U$ is a functorially finite wide subcategory by Lemma \ref{lem:2-3}, thus there are some artin algebra $\Gamma$ and an equivalence $F \colon \WW_U \equi \mod\Gamma$ by Lemma \ref{lem:fwide}.
  This equivalence implies that $|P(\WW_U)| = |\Gamma|$.
  Then $F(\HH_{[U,T]})$ is a torsion class in $\mod\Gamma$ with its $\Ext$-progenerator $FM$, which is a support $\tau$-tilting $\Gamma$-module.
  Theorem \ref{thm:ttt} (3) implies that $F(\HH_{[U,T]})$ is sincere in $\mod\Gamma$ if and only if $FM$ is a $\tau$-tilting $\Gamma$-module, that is, $|FM| = |\Gamma|$.
  Therefore, $\HH_{[U,T]}$ is sincere in $\WW_U$ if and only if $|M| = |FM| = |\Gamma| = |P(\WW_U)|$.
\end{proof}

 To use Lemma \ref{lem:sincere-number}, we will discuss the number of indecomposable $\Ext$-projectives of the heart of a given interval.
For an additive category $\CC$ and a subcategory $\DD$ of $\CC$, we denote by $[\DD]$ the two-sided ideal of $\CC$ consisting of morphisms which factor through objects in $\DD$. Then one can consider the ideal quotient $\CC/[\DD]$.
By abuse of notation, for a subcategory $\CC'$ of $\CC$ (which does not necessarily contain $\DD$), we denote by $\CC'/[\DD]$ the image of the composition of the natural functors $\CC' \hookrightarrow \CC \defl \CC/[\DD]$.
Then the following is a key lemma to compute the number of indecomposable $\Ext$-projectives.
\begin{lemma}\label{lem:number}
  Let $[\UU,\TT]$ be an interval in $\tors\Lambda$ and $\HH$ its heart. Then the following hold.
  \begin{enumerate}
    \item The torsion-free coradical $g$ of $\UU^\perp$ induces a faithful functor $\ov{g} \colon \TT/[\UU] \to \HH$.
    \item If $\TT$ has an $\Ext$-progenerator $T$, then $\ov{g}$ is fully faithful on $(\add T)/[\UU]$, thus induces an equivalence $(\add T)/[\UU] \equi \add (gT)$.
  \end{enumerate}
\end{lemma}
\begin{proof}
  (1)
  We have already shown in Lemma \ref{lem:heartprogen} that $g$ defines a functor $\TT \to \HH$. Let $\varphi \colon X \to Y$ be a morphism with $X,Y \in \TT$.
  We have the following commutative diagram.
  \[
  \begin{tikzcd}
    0 \rar & uX \rar & X \rar\dar["\varphi"] & gX \rar\dar["g(\varphi)"] & 0 \\
    0 \rar & uY \rar["\iota"] & Y \rar & gY \rar & 0
  \end{tikzcd}
  \]
  Since $\iota$ is a right $\UU$-approximation of $Y$, to show that $g$ induces a faithful functor $\ov{g} \colon \TT/[\UU] \to \HH$, it suffices to show that $g(\varphi) = 0$ if and only if $\varphi$ factors through $\iota$. This can be checked straightforwardly.

  (2)
  Let $T_1$ and $T_2$ be in $\add T$. By (1), it suffices to show that $g|_{\add T} \colon \add T \to \HH$ is full.
  Take any map $\psi \colon gT_1 \to gT_2$, and consider the following diagram.
  \[
  \begin{tikzcd}
    0 \rar & uT_1 \rar & T_1 \rar\dar[dashed, "\varphi"] & gT_1 \rar\dar["\psi"] & 0 \\
    0 \rar & uT_2 \rar["\iota"] & T_2 \rar & gT_2 \rar & 0
  \end{tikzcd}
  \]
  Since $uT_2 \in \UU \subseteq\TT$ and $T_1$ is $\Ext$-projective in $\TT$, we obtain a map $\varphi$ which makes the above diagram commute. This shows $g(\varphi) = \psi$, which completes the proof.
\end{proof}
As a corollary, we obtain the following formula for the number of indecomposable $\Ext$-projectives of a heart. For a Krull-Schmidt category $\CC$, we denote by $\ind\CC$ the set of isomorphism classes of indecomposable objects in $\CC$.
\begin{corollary}\label{cor:numproj}
  Let $[\UU,\Fac T]$ be an interval in $\tors\Lambda$ with $T \in \sttilt\Lambda$, and let $\HH$ be its heart with its $\Ext$-progenerator $P(\HH)$ (which exists by Lemma \ref{lem:heartprogen}).
  Then $|P(\HH)| = \# (\ind (\add T) \setminus \ind\UU)$ holds, that is, $|P(\HH)|$ is equal to the number of indecomposable direct summands of $T$ which do not belong to $\UU$.
\end{corollary}
\begin{proof}
  Let $g \colon \mod\Lambda \to \UU^\perp$ be the torsion-free coradical of $\UU^\perp$. By Lemma \ref{lem:heartprogen}, we have that $gT$ is an $\Ext$-progenerator of $\HH$, thus $|P(\HH)| = \#\ind (\add (gT))$ holds.
  Lemma \ref{lem:number} implies that this number is equal to $\#\ind ((\add T)/[\UU])$.

  Set $\TT=\Fac T$ and let $\pi \colon \TT \to \TT/[\UU]$ be the natural projection functor. Since $\UU$ is a subcategory of $\TT$ closed under direct summands and $\TT$ is Krull-Schmidt, it is well-known that $\pi$ induces a bijection
  \[
  \ind\TT \setminus \ind \UU \xrightarrow{\sim} \ind (\TT/[\UU]).
  \]
  Thus by restricting this bijection, we obtain a bijection $\ind(\add T) \setminus \ind\UU \xrightarrow{\sim} \ind ((\add T)/[\UU])$.
\end{proof}

Now we obtain the following first description of $\sint\Lambda$.

\begin{corollary}\label{cor:howtowttilt}
  Let $\Lambda$ be a $\tau$-tilting finite artin algebra. Then $\sint\Lambda$ coincides with the set of intervals $[U,T]$ in $\sttilt\Lambda$ satisfying the following two conditions:
  \begin{itemize}
    \item $U \leq T \leq U^+$ holds.
    \item The number of indecomposable direct summands of $T$ which do not belong to $\Fac U$ is equal to the number of arrows in $\Hasse(\sttilt\Lambda)$ ending at $U$.
  \end{itemize}
\end{corollary}

\begin{proof}
  By Lemma \ref{lem:sincere-number}, an interval $[U,T]$ satisfying $U \leq T \leq U^+$ is a sincere interval if and only if $|P(\WW_U)| = |P(\HH_{[U,T]})|$. The left hand side is equal to the number of arrows in $\Hasse(\sttilt\Lambda)$ ending at $U$ by \cite[Theorem 4.2]{AP}, and the right hand side is equal to the number of indecomposable direct summands of $T$ which do not belong to $\Fac U$ by Corollary \ref{cor:numproj}. Thus the assertion holds.
\end{proof}

Next, we shall give the second method to obtain $\sint\Lambda$, which only uses the poset-theoretic information on $\sttilt\Lambda$. This is based on the following characterization of sincere torsion classes.
For an element $p$ in a poset $P$, we define
\[
^\to p := \{ q \in P \, | \, \text{there is an arrow $p \ot q$ in $\Hasse P$} \}.
\]
\begin{lemma}\label{lem:sincerecri}
  Let $\AA$ be an abelian length category and $\TT$ a torsion class in $\AA$. Then $\TT$ is a sincere torsion class if and only if the following condition is satisfied:
  for any proper subset $X$ of $^\to 0$ in $\tors\AA$, we have
  $\TT \not\subseteq \bigvee X$.
\end{lemma}
\begin{proof}
  Since $\AA$ is length, we can easily check
  \[
  ^\to 0 = \{ \Filt S \, | \, S \in \simp\AA \}.
  \]
  Moreover, for a subset $\SS$ of $\simp\AA$, we have that $\Filt\SS$ is a Serre subcategory of $\AA$ consisting of objects whose all composition factors are in $\SS$. In particular, $\Filt\SS$ is a torsion class, thus $\Filt\SS = \bigvee \{ \Filt S \, | \, S \in \SS \}$.
  Therefore, the described condition is equivalent to that, for any proper subset $\SS \subsetneq \simp\AA$, we have $\TT \not\subseteq \Filt\SS$. This is clearly equivalent to that $\TT$ is sincere.
\end{proof}

Now we obtain the following second description of $\sint\Lambda$.

\begin{corollary}\label{cor:posetstr}
  Let $\Lambda$ be a $\tau$-tilting finite artin algebra. Then $\sint\Lambda$ coincides with the set of intervals $[U,T]$ in $\sttilt\Lambda$ satisfying the two following conditions:
  \begin{itemize}
    \item $U \leq T \leq U^+$ holds.
    \item For any proper subset $X$ of $^\to U$ in $\sttilt\Lambda$, we have $T \not\leq \bigvee X$ holds.
  \end{itemize}
\end{corollary}

\section{The hereditary case}\label{sec:5}

Throughout this section, {\it we denote by} $\Lambda$ {\it a hereditary artin algebra.}  In this section, we investigate ICE-closed subcategories when $\Lambda$ is hereditary. The first author proved in \cite{rigid} that there is a bijection between rigid $\Lambda$-modules and ICE-closed subcategories of $\mod\Lambda$ with enough $\Ext$-projectives. First, we give another proof of this bijection by using wide $\tau$-tilting modules. Next, we study the Hasse quiver of ICE-closed subcategories of $\mod\Lambda$ by introducing a mutation of rigid modules.

\subsection{Rigid modules and ICE-closed subcategories}
We start with introducing the following notation.
\begin{definition}
  Let $\Lambda$ be a hereditary artin algebra.
  \begin{enumerate}
      \item A $\Lambda$-module $T$ is \emph{rigid} if $\Ext_{\Lambda}^{1}(T, T)=0$ holds. We denote by $\rigid\Lambda$ the set of isomorphism classes of basic rigid $\Lambda$-modules.
      \item We denote by $\icep\Lambda$ the set of ICE-closed subcategories of $\mod\Lambda$ with enough $\Ext$-projectives.
  \end{enumerate}
\end{definition}
The aim of this subsection is to give a proof of the following theorem:
\begin{theorem}\label{thm:rigid}\cite[Theorem 2.3]{rigid}
Let $\Lambda$ be a hereditary artin algebra. Then we have mutually inverse bijections
\[
  \begin{tikzcd}
   \rigid\Lambda \rar["\ccok" , shift left] & \icep\Lambda \lar["P" , shift left].
  \end{tikzcd}
  \]
\end{theorem}

We deduce the above theorem from Theorem \ref{thm:wttiltbij}. To this end, we study wide $\tau$-tilting modules and doubly functorially finite ICE-closed subcategories in detail for the hereditary case.

In this section, we will frequently use the following proposition.

\begin{proposition}\label{prop:rigidwide}\cite[Proposition 5.1]{rigid}
  Let $\Lambda$ be a hereditary artin algebra and $T$ a rigid $\Lambda$-module. Then $\la T \ra_{\wide}$ is equivalent to $\mod\Gamma$ for some hereditary artin algebra $\Gamma$, and $T$ is a tilting $\Gamma$-module under the equivalence $\la T\ra_{\wide}\equi\mod\Gamma$.
\end{proposition}

First, we show that $\wttilt\Lambda$ coincides with $\rigid\Lambda$.

\begin{proposition}\label{prop:heredwttilt}
Let $\Lambda$ be a hereditary artin algebra and $T\in\mod\Lambda$. Then the following are equivalent.
\begin{enumerate}
    \item $T$ is a wide $\tau$-tilting module.
    \item $T$ is a rigid $\Lambda$-module.
\end{enumerate}
\end{proposition}
\begin{proof}
  (1) $\Rightarrow$ (2): By Theorem \ref{thm:wttiltbij}, there is an ICE-closed subcategory $\CC$ which has an $\Ext$-progenerator $T$. Then $\Ext_{\Lambda}^{1}(T, T)=0$ holds, that is, $T$ is rigid.

  (2) $\Rightarrow$ (1): This follows from Proposition \ref{prop:rigidwide} since $T$ is tilting (hence $\tau_\WW$-tilting) in $\WW := \la T \ra_{\wide}$.
\end{proof}

Next, we investigate the doubly functorial finiteness of ICE-closed subcategories. For this purpose, the following criterion of the functorial finiteness due to Auslander and Smal\o \ is useful.

\begin{lemma}\cite[Theorem 4.5]{AS}\label{lem:cocover}
  Let $\Lambda$ be an artin algebra and $\CC$ a subcategory of $\mod\Lambda$ which is closed under images. Then the following conditions are equivalent.
  \begin{enumerate}
      \item $\CC$ is contravariantly finite.
      \item There is some $M \in \CC$ satisfying $\CC\subseteq\Sub M$.
  \end{enumerate}
\end{lemma}

\begin{proposition}\label{prop:heredice}
Let $\Lambda$ be a hereditary artin algebra and $\CC$ an ICE-closed subcategory of $\mod\Lambda$. Then the following are equivalent.
\begin{enumerate}
    \item $\CC$ is doubly functorially finite.
    \item $\CC$ is functorially finite.
    \item $\CC$ has an $\Ext$-progenerator.
    \item $\CC$ has enough $\Ext$-projectives.
\end{enumerate}
\end{proposition}
\begin{proof}
  (1) $\Rightarrow$ (2): This follows from Proposition \ref{prop:dfice}.

  (2) $\Rightarrow$ (1): By the dual of \cite[Proposition 4.3]{rigid}, we have $\FFF(\CC)=\Sub\CC$. By Lemma \ref{lem:cocover}, there is some $M\in\CC$ satisfying $\CC\subseteq\Sub M$. Then $\Sub\CC\subseteq\Sub M$ holds, and $\Sub\CC$ is contravariantly finite by the same lemma. Therefore $\FFF(\CC)=\Sub\CC$ is functorially finite. By Proposition \ref{prop:dfice}, we have that $\CC$ is doubly functorially finite.

  (1) $\Rightarrow$ (3): This follows from Theorem \ref{thm:wttiltbij}.

  (3) $\Rightarrow$ (1): It follows from Theorem \ref{thm:classifyice} that $\CC$ is a torsion class in $\la\CC\ra_{\wide}$, the smallest wide subcategory containing $\CC$. Let $P$ be an $\Ext$-progenerator of $\CC$. Then $\ccok P=\CC$ holds by Lemma \ref{lem:icecok}. By Proposition \ref{prop:rigidwide}, we have that $\la P\ra_{{\wide}}$ is a wide subcategory of $\mod\Lambda$ which is equivalent to $\mod\Gamma$ for some artin algebra $\Gamma$, and $\la P\ra_{{\wide}}$ is functorially finite in $\mod\Lambda$ by Lemma \ref{lem:fwide}. Clearly $\la P\ra_{{\wide}}=\la\CC\ra_{\wide}$ holds by $\ccok P=\CC$. By Theorem \ref{thm:ttt} (3) (applied to $\mod\Gamma$), we conclude that $\CC$ is functorially finite in $\la P\ra_{{\wide}}$.

  (3) $\Rightarrow$ (4): Trivial.

  (4) $\Rightarrow$ (3): For any $\Ext$-projective object $P$ in $\CC$, we have $|P|\leq|\Lambda|$ since $P$ is a partial tilting $\Lambda$-module (see e.g. \cite[VI.2]{ASS}). Thus there are at most $|\Lambda|$ indecomposable $\Ext$-projective objects in $\CC$ up to isomorphism. Thus a direct sum of all non-isomorphic indecomposable $\Ext$-projective objects in $\CC$ is an $\Ext$-progenerator of $\CC$.
\end{proof}

\begin{proof}[Proof of Theorem \ref{thm:rigid}]
  This immediately follows from Theorem \ref{thm:wttiltbij}, Propositions \ref{prop:heredwttilt} and \ref{prop:heredice}.
\end{proof}

\subsection{Mutations of rigid modules}
In this subsection, we investigate the Hasse quiver of $\dfice\Lambda$. The bijection in Theorem \ref{thm:rigid} induces a partial order on $\rigid\Lambda$, that is, for $U$ and $T$ in $\rigid\Lambda$, we define $U\leq T$ if and only if $\ccok U\subseteq\ccok T$. Then we want to study when there is an arrow $T\to U$ in $\Hasse(\rigid\Lambda)$.

We start with defining an operation $\mu$ which is an analog of $\tau$-tilting mutation introduced in \cite{AIR}.
\begin{definition}\label{def:mu}
  Let $T$ be a basic rigid $\Lambda$-module. For an indecomposable direct summand $X$ of $T$, we define a module $\mu_{X}(T)$ as follows. Let $T=X\oplus U$, and take an exact sequence
  \[
  \begin{tikzcd}[row sep=0]
     X \rar["f"] & U' \rar & Y \rar & 0,
  \end{tikzcd}
  \]
  where $f$ is a minimal left $\add U$-approximation.
  Define $\mu_{X}(T)$ as the unique basic module satisfying $\add \mu_X(T) = \add(Y \oplus U)$.
\end{definition}

Thanks to the hereditary property, we have the following useful results.

\begin{lemma}\cite[Lemma 3.2.]{rigid}\label{lem:imclosed}
 Let $\Lambda$ be a hereditary artin algebra and $U$ a rigid $\Lambda$-module. Then $\Fac U\cap \Sub U=\add U$ holds.
\end{lemma}

This lemma implies that $\add U$ is closed under images.

\begin{lemma}\label{lem:surjorinj}
  In Definition \ref{def:mu}, we have that $f$ is either surjective or injective.
\end{lemma}
\begin{proof}
  Suppose that $f$ is not injective. By Lemma \ref{lem:imclosed}, we have $\im f\in\add T$. Since $\im f$ is a proper factor of $X$, it has no direct summands isomorphic to $X$ and thus belongs to $\add U$. Then the canonical map $X\to\im f$ is also a minimal left $\add U$-approximation and surjective. Then $f$ is surjective by the uniqueness of a minimal left approximation.
\end{proof}
The following proposition implies that $\mu$ defines an operation on $\rigid\Lambda$.
\begin{proposition}
  In Definition \ref{def:mu}, we have that $\mu_{X}(T)$ is a rigid $\Lambda$-module.
\end{proposition}
\begin{proof}
  If $f$ in Definition \ref{def:mu} is surjective, then $Y=0$ and $\mu_{X}(T)=U$ is obviously rigid. Assume that $f$ is not surjective. Then $f$ is injective by Lemma \ref{lem:surjorinj}, hence the short exact sequence in Definition \ref{def:mu} is as follows:
  \[
  \begin{tikzcd}[row sep=0]
     0 \rar & X \rar["f"] & U' \rar & Y \rar & 0
  \end{tikzcd}
  \]
  We will show that $\Ext_{\Lambda}^{1}(U, Y), \Ext_{\Lambda}^{1}(Y, U)$ and $\Ext_{\Lambda}^{1}(Y, Y)$ vanish. Since $\Lambda$ is hereditary, $\Ext_{\Lambda}^{1}(U, -)$ preserves surjections. Applying $\Ext_{\Lambda}^{1}(U, -)$ to the above sequence, we get $\Ext_{\Lambda}^{1}(U, Y)=0$ since $\Ext_{\Lambda}^{1}(U, U)=0$. Applying $\Hom_{\Lambda}(-, U)$ to the sequence, we obtain the following exact sequence:
  \[
  \begin{tikzcd}[row sep=0]
     \Hom_{\Lambda}(U', U) \rar["f^{*}"] & \Hom_{\Lambda}(X, U) \rar & \Ext_{\Lambda}^{1}(Y, U) \rar & \Ext_{\Lambda}^{1}(U', U)
  \end{tikzcd}
  \]
  Here $\Ext_{\Lambda}^{1}(U', U)=0$ by $U'\in\add U$. Since $f$ is a left $\add U$-approximation, we have that $f^{*}$ is surjective. Thus $\Ext_{\Lambda}^{1}(Y, U)=0$ holds. Then $\Ext_{\Lambda}^{1}(Y, U')=0$ holds. Hence $\Ext_{\Lambda}^{1}(Y, Y)=0$ holds because $\Ext_{\Lambda}^{1}(Y, -)$ preserves surjections.
\end{proof}

The following proposition is convenient for calculating $\mu$.

\begin{proposition}\label{prop:fac}
Let $T$ be a basic rigid $\Lambda$-module and $X$ an indecomposable direct summand of $T$. Set $T=X\oplus U$. If $X\in\Fac U$, then $\mu_{X}(T)=U$ holds.
\end{proposition}
\begin{proof}
  Take a map $f$ in Definition \ref{def:mu}. By Lemma \ref{lem:surjorinj}, $f$ is either surjective or injective. If $f$ is injective, then $X\in\Fac U\cap\Sub U$ holds. This implies $X\in\add U$ by Lemma \ref{lem:imclosed}, which is a contradiction. Thus $f$ is surjective and the assertion holds immediately.
\end{proof}

Our aim is to show that there is an arrow $T\to U$ in $\Hasse(\rigid\Lambda)$ if and only if $U\iso\mu_{X}(T)$ holds for some indecomposable direct summand $X$ of $T$. For this purpose, we need the following lemmas.
\begin{lemma}\label{lem:neq}
  Let $T$ be a basic rigid $\Lambda$-module and $X$ an indecomposable direct summand of $T$. Then $\ccok\mu_{X}(T)\subsetneq\ccok T$ holds.
\end{lemma}
\begin{proof}
  By the definition of $\mu$, we have $\mu_{X}(T)\in\ccok T$. Since $\ccok T$ is closed under cokernels by Theorem \ref{thm:rigid}, we have $\ccok\mu_{X}(T)\subseteq\ccok T$. Let $T=X\oplus U$, and assume that $\ccok\mu_{X}(T)=\ccok T$. By the definition of $\mu$, we have $\mu_{X}(T)\in\Fac U$. Therefore we have $\ccok\mu_{X}(T)\subseteq\Fac U$ since $\Fac U$ is closed under quotients. Hence $X$ is contained in $\ccok T=\ccok\mu_{X}(T)\subseteq\Fac U$. Then $\mu_{X}(T)=U$ holds by Proposition \ref{prop:fac}. On the other hand, we have $\mu_{X}(T)=T$ by $\ccok\mu_{X}(T)=\ccok T$ and Theorem \ref{thm:rigid}. Thus we have $U=\mu_{X}(T)=T$, which is a contradiction. Thus $\ccok\mu_{X}(T)\subsetneq\ccok T$ holds.
\end{proof}
\begin{lemma}\label{lem:rigid1}
  Let $T$ be a basic rigid $\Lambda$-module and $X$ and $Y$ indecomposable direct summands of $T$ which are not isomorphic to each other. Then $\ccok\mu_{X}(T)\nsubseteq\ccok\mu_{Y}(T)$ holds.
\end{lemma}
\begin{proof}
  Set $T=Y\oplus U$, and assume $\ccok\mu_{X}(T)\subseteq\ccok\mu_{Y}(T)$. Then $Y\in\ccok\mu_{Y}(T)$ holds because $Y$ is a direct summand of $\mu_{X}(T)$. Since $Y$ and $U$ belong to $\ccok\mu_{Y}(T)$, so does $T$. Then $\ccok T\subseteq\ccok\mu_{Y}(T)$ holds because $\ccok\mu_{Y}(T)$ is closed under cokernels by Theorem \ref{thm:rigid}. This contradicts Lemma \ref{lem:neq}.
\end{proof}

\begin{lemma}\label{lem:rigid2}
  Let $T$ be a basic rigid $\Lambda$-module. For any rigid module $T'$ which satisfies $\ccok T'\subsetneq\ccok T$, there is an indecomposable direct summand $X$ of $T$ satisfying $\ccok T'\subseteq\ccok\mu_{X}(T)$.
\end{lemma}
\begin{proof}
  It suffices to show that there is an indecomposable direct summand $X$ of $T$ satisfying $T'\in\ccok\mu_{X}(T)$ because $\ccok\mu_{X}(T)$ is closed under cokernels by Theorem \ref{thm:rigid}.

  Set $\WW=\la T\ra_{\wide}$, the smallest wide subcategory containing $T$. By Proposition \ref{prop:rigidwide}, there is a hereditary artin algebra $\Gamma$ such that $\WW$ is equivalent to $\mod\Gamma$, and $T$ is a tilting $\Gamma$-module under the equivalence.
  Then we have $T'\in\ccok T'\subsetneq\ccok T\subseteq\WW$ because $\WW$ is closed under cokernels. Since $T'$ is rigid and $\Gamma$ is hereditary, we have that $T'$ is a partial tilting $\Gamma$-module under the equivalence $\WW\iso\mod\Gamma$.
  Let $\Fac_{\WW}$ denote $\Fac$ considered in the abelian category $\WW$. Then $\Fac_{\WW}T'$ and $\Fac_{\WW}T$ are torsion classes in $\WW$. Since $T$ is a tilting $\Gamma$-module, we have that $T$ is an $\Ext$-progenerator of $\Fac_{\WW}T$ and $\ccok T=\Fac_{\WW}T$ holds. Therefore, we have $\Fac_{\WW}T'\subseteq\Fac_{\WW}T$ because we have $T'\in\ccok T=\Fac_{\WW}T$.

  Assume $\Fac_{\WW}T'=\Fac_{\WW}T$. Since $T'$ is rigid and $\Ext_{\Lambda}^{1}(T', -)$ preserves surjections, we have that $T'$ is $\Ext$-projective in $\Fac_{\WW}T'=\Fac_{\WW}T$. Hence $T'$ is a direct summand of $T$ because $T$ is an $\Ext$-progenerator of $\Fac_{\WW}T$. In this case, take an indecomposable direct summand $X$ of $T$ which does not belong to $\add T'$. Then we have $T'\in\ccok\mu_{X}(T)$, and $X$ is a desired direct summand of $T$.

  Assume $\Fac_{\WW}T'\subsetneq\Fac_{\WW}T$. By \cite[Theorems 2.30, 2.35]{AIR}, there exists an indecomposable direct summand $X$ of $T$ such that $\Fac_{\WW}T'\subseteq\Fac_{\WW}\mu_{X}^{-}(T)$ with $\mu^{-}$ in \cite[Definition-Proposition 2.28]{AIR}. Then $\mu_{X}^{-}(T)$ coincides with $\mu_{X}(T)$ in Definition \ref{def:mu} because $\add U$ is contained in $\WW$ and $\add U$-approximations in $\mod\Lambda$ and those in $\WW$ are the same. Now we have $\Fac_{\WW}\mu_{X}(T)=\ccok\mu_{X}(T)$ since $\mu_{X}(T)$ is an $\Ext$-progenerator of $\Fac_{\WW}\mu_{X}(T)$. Thus we have $T'\in\Fac_{\WW}T'\subseteq\Fac_{\WW}\mu_{X}(T)=\ccok\mu_{X}(T)$, and $X$ is a desired direct summand of $T$.
\end{proof}
Now we obtain the main result of this subsection.
\begin{theorem}\label{thm:mu}
  Let $T$ be a basic rigid $\Lambda$-module and $X$ an indecomposable direct summand of $T$. Then we have an arrow $T \to \mu_X(T)$ in $\Hasse(\rigid\Lambda)$.
\end{theorem}
\begin{proof}
  Lemma \ref{lem:neq} shows $\ccok \mu_X(T) \subsetneq \ccok T$.
  Assume that there is a rigid $\Lambda$-module $T'$ satisfying $\ccok\mu_{X}(T)\subsetneq\ccok T'\subsetneq\ccok T$. By Lemma \ref{lem:rigid2}, there is an indecomposable direct summand $Y$ of $T$ satisfying $\ccok T'\subseteq\ccok\mu_{Y}(T)$. Then $\ccok\mu_{X}(T)\subsetneq\ccok\mu_{Y}(T)$ holds. This contradicts Lemma \ref{lem:rigid1}.
\end{proof}

As a result of this section, we obtain the following corollary.
\begin{corollary}\label{cor:mu}
  Let $T$ and $U$ be basic rigid $\Lambda$-modules. The following conditions are equivalent.
  \begin{enumerate}
      \item There is an arrow $T\to U$ in $\Hasse(\rigid\Lambda)$.
      \item $\mu_{X}(T)\iso U$ holds for some indecomposable direct summand $X$ of $T$.
  \end{enumerate}
\end{corollary}
\begin{proof}
  (1) $\Rightarrow$ (2): Since we have $\ccok U \subsetneq \ccok T$, Lemmas \ref{lem:neq} and \ref{lem:rigid2} imply that there is some indecomposable direct summand $X$ of $T$ satisfying $\ccok U \subseteq \ccok \mu_X(T) \subsetneq \ccok T$.
  Then $\ccok U = \ccok \mu_X(T)$ holds since we have a Hasse arrow $T \to U$. This implies $U = \mu_X(T)$ by Theorem \ref{thm:rigid}.

  (2) $\Rightarrow$ (1): This follows from Theorem \ref{thm:mu}.
\end{proof}

We obtain the following combinatorial consequence on $\Hasse(\rigid\Lambda) \iso \Hasse(\dfice\Lambda)$.
\begin{corollary}\label{cor:arrow}
  Let $T$ be a basic rigid $\Lambda$-module. Then there are exactly $|T|$ arrows starting at $T$ in $\Hasse(\rigid\Lambda)$.
\end{corollary}

We end this section by showing a typical example.

\begin{example}\label{ex:a3}
  Let $\Lambda=k[1\leftarrow2\leftarrow3]$. Since $\Lambda$ is a  representation-finite hereditary algebra, we have $\dfice\Lambda=\ice\Lambda = \icep\Lambda$ by Propositions \ref{prop:ttfice} and \ref{prop:heredice}, and we have a bijection $\ccok \colon \rigid\Lambda \to \ice\Lambda$ by Theorem \ref{thm:rigid}.
  The Hasse quiver of $\rigid\Lambda$ is depicted in Figure \ref{hasseA3}.
  The red vertices are wide $\tau$-tilting modules which are not support $\tau$-tilting modules.
  \begin{figure}[htp]
    \begin{tikzpicture}
      \node[block=3] (121321) at (0,10) {\nodepart{one}$\sst{1}$\nodepart{two}$\sst{2 \\ 1}$\nodepart{three}$\sst{3 \\ 2 \\ 1}$};
      \node[block=2] (121) at (-4,9) {\nodepart{one}$\sst{1}$\nodepart{two}$\sst{2 \\ 1}$};
      \node[block=3] (221321) at (0,9) {\nodepart{one}$\sst{2}$\nodepart{two}$\sst{2 \\ 1}$\nodepart{three}$\sst{3 \\ 2 \\ 1}$};
      \node[block=3] (13321) at (4,9) {\nodepart{one}$\sst{1}$\nodepart{two}$\sst{3}$\nodepart{three}$\sst{3 \\ 2 \\ 1}$};
      \node[block=2] (221) at (-3,8) {\nodepart{one}$\sst{2}$\nodepart{two}$\sst{2 \\ 1}$};
      \node[block=2,red] (21321) at (-1,8) {\nodepart{one}$\sst{2 \\ 1}$\nodepart{two}$\sst{3 \\ 2 \\ 1}$};
      \node[block=3] (232321) at (1,8) {\nodepart{one}$\sst{2}$\nodepart{two}$\sst{3 \\ 2}$\nodepart{three}$\sst{3 \\ 2 \\ 1}$};
      \node[rectangle,draw,red] (21) at (-2,7) {\nodepart{one}$\sst{2 \\ 1}$};
      \node[block=2,red] (2321) at (0,7) {\nodepart{one}$\sst{2}$\nodepart{two}$\sst{3 \\ 2 \\ 1}$};
      \node[block=2] (232) at (1.5,7) {\nodepart{one}$\sst{2}$\nodepart{two}$\sst{3 \\ 2}$};
      \node[block=3] (332321) at (3,7) {\nodepart{one}$\sst{3}$\nodepart{two}$\sst{3 \\ 2}$\nodepart{three}$\sst{3 \\ 2 \\ 1}$};
      \node[block=2,red] (1321) at (5,7) {\nodepart{one}$\sst{1}$\nodepart{two}$\sst{3 \\ 2 \\ 1}$};
      \node[block=2] (13) at (6,2) {\nodepart{one}$\sst{1}$\nodepart{two}$\sst{3}$};
      \node[rectangle,draw] (2) at (-3,5) {\nodepart{one}$\sst{2}$};
      \node[block=2,red] (3321) at (-1,5) {\nodepart{one}$\sst{3}$\nodepart{two}$\sst{3 \\ 2 \\ 1}$};
      \node[block=2] (332) at (2,5) {\nodepart{one}$\sst{3}$\nodepart{two}$\sst{3 \\ 2}$};
      \node[block=2,red] (32321) at (4,5) {\nodepart{one}$\sst{3 \\ 2}$\nodepart{two}$\sst{3 \\ 2 \\ 1}$};
      \node[rectangle,draw,red] (321) at (0,3) {\nodepart{one}$\sst{3 \\ 2 \\ 1}$};
      \node[rectangle,draw,red] (32) at (3,3) {\nodepart{one}$\sst{3 \\ 2}$};
      \node[rectangle,draw] (3) at (0,1) {\nodepart{one}$\sst{3}$};
      \node[rectangle,draw] (1) at (-4,0.5) {\nodepart{one}$\sst{1}$};
      \node (0) at (0,0) {$\sst{0}$};

      \draw[->] (121321) -- (121);
      \draw[->] (121321) -- (221321);
      \draw[->] (121321) -- (13321);
      \draw[->] (121) -- (221);
      \draw[->] (121) -- (1);
      \draw[->] (221321) -- (221);
      \draw[->] (221321) -- (21321);
      \draw[->] (221321) -- (232321);
      \draw[->] (13321) -- (332321);
      \draw[->] (13321) -- (1321);
      \draw[->] (13321) .. controls +(right:16mm) ..  (13);
      \draw[->] (221) -- (2);
      \draw[->] (221) -- (21);
      \draw[->] (21321) -- (21);
      \draw[->] (21321) -- (3321);
      \draw[->] (232321) -- (2321);
      \draw[->] (232321) -- (232);
      \draw[->] (232321) -- (332321);
      \draw[->] (2321) .. controls +(left:10mm) ..  (2);
      \draw[->] (2321) -- (321);
      \draw[->] (232) -- (332);
      \draw[->] (232) -- (2);
      \draw[->] (332321) -- (3321);
      \draw[->] (332321) -- (332);
      \draw[->] (332321) -- (32321);
      \draw[->] (1321) -- (32321);
      \draw[->] (1321) .. controls +(down:43mm) ..  (1);
      \draw[->] (13) .. controls +(down:10mm) .. (1);
      \draw[->] (13) -- (3);
      \draw[->] (3321) -- (321);
      \draw[->] (3321) -- (3);
      \draw[->] (332) -- (3);
      \draw[->] (332) -- (32);
      \draw[->] (32321) -- (321);
      \draw[->] (32321) -- (32);
      \draw[->] (1) .. controls +(down:5mm) ..  (0);
      \draw[->] (2) .. controls +(down:20mm) .. (0);
      \draw[->] (21) .. controls +(down:20mm) .. (0);
      \draw[->] (3) -- (0);
      \draw[->] (321) .. controls +(-45:10mm) and +(45:10mm) .. (0);
      \draw[->] (32) .. controls +(down:20mm) .. (0);
    \end{tikzpicture}
    \caption{$\Hasse(\rigid kQ)$ for $Q: 1 \ot 2 \ot 3$}
    \label{hasseA3}
  \end{figure}
\end{example}

\section{Hasse quivers of ICE-closed subcategories}\label{sec:6}
In this section, we prove some relations between the lattice structure of $\ice\AA$ and $\tors\AA$, and then give two examples of computations of $\ice\Lambda$ and $\Hasse(\ice\Lambda)$ via wide $\tau$-tilting modules.
At the end of this section, we raise two questions on $\Hasse(\ice\Lambda)$ which naturally arise from the computations of the examples.

\subsection{Some properties on the lattice of ICE-closed subcategories}
Let $\WW$ be a wide subcategory of $\AA$. Then we have $\tors\WW \subseteq \ice\AA$ by Lemma \ref{lem:torsice}, and clearly $\tors\WW$ is a full subposet of $\ice\AA$. Moreover, we can prove a stronger statement:
\begin{proposition}\label{prop:sublat}
  Let $\WW$ be a wide subcategory of $\AA$. Then $\tors\WW$ is a closed sublattice of $\ice\AA$, that is, for every non-empty subset $\{\TT_i\} _{i \in I}$ of $\tors\WW$, its join and meet in $\ice\AA$ are both in $\tors\WW$.
\end{proposition}
\begin{proof}
  Since the meet of $\{\TT_i\}_{i \in I}$ in $\ice\AA$ is just a set-theoretic intersection, it clearly belongs to $\tors\WW$.
  It suffices to show that its join in $\tors\WW$ is actually a join in $\ice\AA$. Indeed, one can compute the join in $\tors\WW$ as follows, where the subscript $_\WW$ means that we consider operations inside the abelian category $\WW$:
  \[
    \bigvee \{ \TT_i\}_{i \in I} = \TTT_\WW (\bigcup_{i \in I}\TT_i) = \Filt\Fac_\WW(\bigcup_{i\in I}\TT_i)
    =\Filt(\bigcup_{i\in I}\TT_i).
  \]
  The last equality follows from the fact that $\TT_i$ is closed under quotients in $\WW$.
  Let $\CC$ be an ICE-closed subcategory containing $\TT_i$ for every $i \in I$. Then clearly it contains $\bigcup_{i \in I} \TT_i$, thus contains $\Filt(\bigcup_{i \in I} \TT_i)$ since $\CC$ is extension-closed.
  Therefore, $\Filt(\bigcup_{i \in I} \TT_i)$ is a join of $\{ \TT_i \}_{i \in I}$ in $\ice\AA$.
\end{proof}

As an application, we deduce the following combinatorial property on $\Hasse(\ice\AA)$.
\begin{corollary}
  Let $\CC$ be an ICE-closed subcategory of $\AA$ which is \emph{not} a torsion class. Then in $\Hasse(\ice\AA)$, there is at most one arrow of the form $\CC \to \TT$ with $\TT \in \tors\AA$, and there is at most one arrow of the form $\TT \to \CC$ with $\TT \in \tors\AA$.
\end{corollary}
\begin{proof}
  Assume that there are two arrows $\CC \to \TT_i$ with $\TT_i \in \tors\AA$ for $i=1,2$. Then we clearly have $\CC = \TT_1 \vee \TT_2$ in $\ice\AA$, which implies $\CC \in \tors\AA$ by Proposition \ref{prop:sublat}. The same proof applies to the latter assertion.
\end{proof}
In addition, we can prove the following result on the Hasse quivers.
\begin{proposition}\label{prop:asfullsub}
  Let $\WW$ be a wide subcategory of $\AA$. Then $\Hasse(\ice\AA)$ contains $\Hasse(\tors\WW)$ as a full subquiver.
\end{proposition}
\begin{proof}
  Let $\UU \ot \TT$ be an arrow in $\Hasse(\tors\WW)$, and suppose that an ICE-closed subcategory $\CC$ of $\AA$ satisfies $\UU\subsetneq\CC\subseteq\TT$. Then it is enough to show $\CC = \TT$.

  First, we give a proof in the case of $\WW=\AA$.
  Since there is an arrow $\UU \ot \TT$ in $\Hasse(\tors\AA)$, there is a brick $B$ satisfying $\HH_{[\UU, \TT]}=\Filt B$ by Proposition \ref{prop:label}, and we have $\TT = \UU * (\Filt B)$ by Lemma \ref{lem:interval}.
  Since $\CC$ contains $\UU$ and is extension-closed, it suffices to show $B \in \CC$ to deduce $\CC = \TT$.
  Take an object $C$ in $\CC$ which is not contained in $\UU$. Since $\TT = \UU * (\Filt B)$, we have the following short exact sequence
  \[
  \begin{tikzcd}
    0 \rar & U \rar & C \rar & B' \rar & 0
  \end{tikzcd}
  \]
  with $U\in\UU$ and $0 \neq B'\in\Filt B$. Since $\CC$ is closed under cokernels, $B'$ is contained in $\CC$. Since $B$ is a brick, $\Filt B$ is an abelian length category with the unique simple object $B$ (see e.g. \cite[1.2]{ringel}), thus we can easily construct an endomorphism $\varphi \colon B' \to B'$ satisfying $\im\varphi \iso B$ by $B' \neq 0$.
  Since $\CC$ is closed under images, $B$ is contained in $\CC$.

  Now consider the general case. Recall that we have $\UU\subsetneq\CC\subseteq\TT$ with $\UU,\TT \in \tors\WW$ and $\CC \in \ice\AA$. Then $\CC$ is contained in $\WW$ by $\CC \subseteq \TT \subseteq \WW$, and is also an ICE-closed subcategory of $\WW$. Since $\WW$ is an abelian length category, the previous argument implies $\CC = \TT$.
\end{proof}
Therefore, roughly speaking, $\Hasse(\ice\AA)$ is obtained by gluing the Hasse quivers $\Hasse(\tors\WW)$ for all wide subcategories $\WW$ of $\AA$. However, we do not know a precise way to glue them.

\subsection{Examples of computations on ICE-closed subcategories}\label{sec:ex}
In this subsection, we give some examples of computations concerning wide $\tau$-tilting modules and ICE-closed subcategories for $\tau$-tilting finite algebras. Throughout this section, we denote by $k$ a field.

First of all, for a path algebra $k Q$ over a Dynkin quiver $Q$, we can completely determine $\wttilt kQ = \rigid kQ$ and $\Hasse(\rigid kQ) \iso \Hasse (\dfice kQ) = \Hasse(\ice kQ)$ by using iterated mutations of rigid modules (Section \ref{sec:5}) starting from $kQ$, see Examples \ref{ex:intro} and \ref{ex:a3} for example.

\begin{example}\label{ex:nak}
  Consider the following algebra
  \[
  \Lambda : = k [ 1 \xleftarrow{\be} 2 \xleftarrow{\al} 3]/(\al \be).
  \]
  The Auslander-Reiten quiver of $\mod \Lambda$ is as follows.
  \[
  \begin{tikzpicture}[scale=.6]
   \node (1) at (0,0) {$\sst{1}$};
   \node (21) at (1,1) {$\sst{2 \\ 1}$};
   \node (2) at (2,0) {$\sst{2}$};
   \node (32) at (3,1) {$\sst{3 \\ 2}$};
   \node (3) at (4,0) {$\sst{3}$};

   \draw[->] (1) -- (21);
   \draw[->] (21) -- (2);
   \draw[->] (2) -- (32);
   \draw[->] (32) -- (3);

   \draw[dashed] (1) -- (2);
   \draw[dashed] (2) -- (3);
  \end{tikzpicture}
  \]
  Since $\Lambda$ is representation-finite, it is $\tau$-tilting finite and we have a bijection $\sttilt\Lambda \rightleftarrows \tors\Lambda$. Figure \ref{fig:naktors} is $\Hasse(\tors\Lambda)$ together with the brick labeling (Proposition \ref{prop:label}), where each vertex represents the corresponding support $\tau$-tilting modules.
  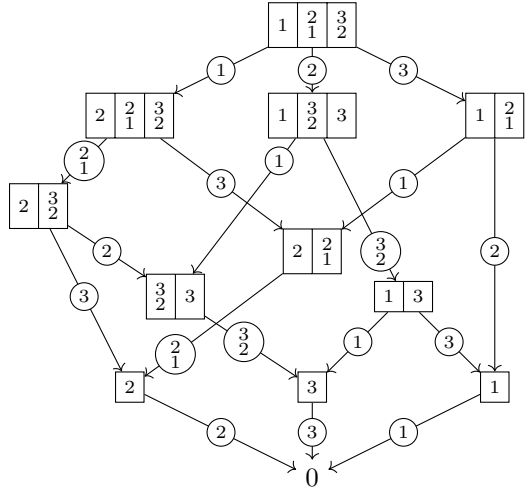
\begin{figure}[htp]
    \begin{tikzpicture}[scale = .6]
      \node (0) at (0,0) {0};
      \node[rectangle,draw] (1) at (4,2) {$\sst{1}$};
      \node[rectangle,draw] (2) at (-4,2) {$\sst{2}$};
      \node[rectangle,draw] (3) at (0,2) {$\sst{3}$};
      \node[block=2] (323) at (-3,4) {\nodepart{one}$\sst{3 \\ 2}$\nodepart{two}$\sst{3}$};
      \node[block=2] (221) at (0,5) {\nodepart{one}$\sst{2}$\nodepart{two}$\sst{2 \\ 1}$};
      \node[block=2] (13) at (2,4) {\nodepart{one}$\sst{1}$\nodepart{two}$\sst{3}$};
      \node[block=2] (232) at (-6,6) {\nodepart{one}$\sst{2}$\nodepart{two}$\sst{3 \\ 2}$};
      \node[block=3] (22132) at (-4,8) {\nodepart{one}$\sst{2}$\nodepart{two}$\sst{2 \\ 1}$\nodepart{three}$\sst{3 \\ 2}$};
      \node[block=3] (1323) at (0,8) {\nodepart{one}$\sst{1}$\nodepart{two}$\sst{3 \\ 2}$\nodepart{three}$\sst{3}$};
      \node[block=2] (121) at (4,8) {\nodepart{one}$\sst{1}$\nodepart{two}$\sst{2 \\ 1}$};
      \node[block=3] (12132) at (0,10) {\nodepart{one}$\sst{1}$\nodepart{two}$\sst{2 \\ 1}$\nodepart{three}$\sst{3 \\ 2}$};

      \node[blabel] (1b0) at (2,1) {$\sst{1}$};
      \node[blabel] (2b0) at (-2,1) {$\sst{2}$};
      \node[blabel] (323b3) at (-1.5,3) {$\sst{3 \\ 2}$};
      \node[blabel] (13b3) at (1,3) {$\sst{1}$};
      \node[blabel] (13b1) at (3,3) {$\sst{3}$};
      \node[blabel] (3b0) at (0,1) {$\sst{3}$};
      \node[blabel] (221b2) at (-3,2.725) {$\sst{2 \\ 1}$};
      \node[blabel] (232b2) at (-5,4) {$\sst{3}$};
      \node[blabel] (232b323) at (-4.5,5) {$\sst{2}$};
      \node[blabel] (22132b232) at (-5,7) {$\sst{2 \\ 1}$};
      \node[blabel] (22132b221) at (-2,6.5) {$\sst{3}$};
      \node[blabel] (1323b323) at (-0.725,7) {$\sst{1}$};
      \node[blabel] (1323b13) at (1.5,5) {$\sst{3 \\ 2}$};
      \node[blabel] (121b221) at (2,6.5) {$\sst{1}$};
      \node[blabel] (121b1) at (4,5) {$\sst{2}$};
      \node[blabel] (12132b22132) at (-2,9) {$\sst{1}$};
      \node[blabel] (12132b1323) at (0,9) {$\sst{2}$};
      \node[blabel] (12132b121) at (2,9) {$\sst{3}$};

      \draw[-] (12132) -- (12132b22132);
      \draw[-] (12132) -- (12132b1323);
      \draw[-] (12132) -- (12132b121);
      \draw[-] (22132) -- (22132b232);
      \draw[-] (22132) -- (22132b221);
      \draw[-] (1323) -- (1323b323);
      \draw[-] (1323) -- (1323b13);
      \draw[-] (121) -- (121b221);
      \draw[-] (121) -- (121b1);
      \draw[-] (232) -- (232b323);
      \draw[-] (232) -- (232b2);
      \draw[-] (221) -- (221b2);
      \draw[-] (323) -- (323b3);
      \draw[-] (13) -- (13b1);
      \draw[-] (13) -- (13b3);
      \draw[-] (2) -- (2b0);
      \draw[-] (3) -- (3b0);
      \draw[-] (1) -- (1b0);

      \draw[->] (12132b22132) -- (22132);
      \draw[->] (12132b1323) -- (1323);
      \draw[->] (12132b121) -- (121);
      \draw[->] (22132b232) -- (232);
      \draw[->] (22132b221) -- (221);
      \draw[->] (1323b323) -- (323);
      \draw[->] (1323b13) -- (13);
      \draw[->] (121b221) -- (221);
      \draw[->] (121b1) -- (1);
      \draw[->] (232b323) -- (323);
      \draw[->] (232b2) -- (2);
      \draw[->] (221b2) -- (2);
      \draw[->] (323b3) -- (3);
      \draw[->] (13b1) -- (1);
      \draw[->] (13b3) -- (3);
      \draw[->] (2b0) -- (0);
      \draw[->] (3b0) -- (0);
      \draw[->] (1b0) -- (0);
    \end{tikzpicture}
    \caption{$\Hasse(\tors \Lambda)$ with the brick labeling, where the vertices are the corresponding support $\tau$-tilting modules}
    \label{fig:naktors}
  \end{figure}

  First, let us consider ICE intervals in $\tors\Lambda$ using Theorem \ref{thm:b}. For example, let $\UU:= \Fac U$ for $U= \sst{3\\2} \oplus \sst{3}$. Then we obtain $\UU^+ = \mod \Lambda$ from $\Hasse(\tors\Lambda)$, thus $[\UU,\TT]$ for each $\TT \in [\UU,\mod\Lambda]$ is an ICE interval.
  Now put $\TT:=\Fac T$ for $T= \sst{2} \oplus \sst{2\\1} \oplus \sst{3\\2}$.
  Then $\HH_{[\UU,\TT]}$ is an ICE-closed subcategory, and Proposition \ref{prop:brickseq} shows $\HH_{[\UU,\TT]} = \Filt \{\sst{2}, \sst{2\\1} \}$ by reading labels, which is equal to $\add \{\sst{2}, \sst{2\\1} \}$.

  Next, let us compute a wide $\tau$-tilting module using Corollary \ref{cor:fromttint}. For these $U$ and $T$, it is easy to see $T/\tr_U(T) = \sst{2} \oplus \sst{2\\1}$, thus this is a wide $\tau$-tilting module corresponding to $\HH_{[\UU,\TT]}$.

  On the other hand, consider $U':= 0$ and $T':= \sst{2} \oplus \sst{2\\1}$. We can check that $[\Fac U',\Fac T']$ is an ICE interval similarly, and its heart is $\Filt \{\sst{2}, \sst{2\\1} \}$ by reading labels. Thus it coincides with $\HH_{[\UU,\TT]}$.
  Actually, we have $T'/\tr_{U'} (T') = T' =  \sst{2} \oplus \sst{2\\1}$, thus both intervals $[\UU,\TT]$ and $[\Fac U',\Fac T']$ give the same wide $\tau$-tilting module and ICE-closed subcategory.
  Now consider whether these intervals are sincere, see Definition \ref{def:sincere-int}.
  We can check that $[U,T]$ satisfies the condition in Corollary \ref{cor:howtowttilt}, but $[U',T']$ does not, hence $[U,T] \in \sint\Lambda$ and $[U',T'] \not\in \sint\Lambda$. This can also be checked by using Corollary \ref{cor:posetstr}.

  Let us look at non-examples. Consider the interval $[\Fac \sst{3}, \Fac (\sst{2} \oplus \sst{3\\2})]$. We have $(\Fac \sst{3})^+ = \Fac (\sst{1} \oplus \sst{3\\2} \oplus \sst{3})$, and it does not contain $\Fac (\sst{2} \oplus \sst{3\\2})$.
  Thus this interval is not an ICE interval by Theorem \ref{thm:b}. In fact, its heart is $\Filt\{ \sst{2},\sst{3\\2}\}$ by Proposition \ref{prop:brickseq}, which is equal to $\add \{ \sst{2},\sst{3\\2}\}$.
  Then the cokernel of the natural embedding $\sst{2} \hookrightarrow \sst{3\\2}$ does not belong this category, thus it is not ICE-closed.

  Now we can compute all the wide $\tau$-tilting modules using Corollary \ref{cor:fromttint} or \ref{cor:howtowttilt}. By calculating $\ccok$ of them or $\Filt$ of labels, we obtain the Hasse quiver of $\ice\Lambda$, Figure \ref{fig:nakice}. In this figure, the red vertices are ICE-closed subcategories which are not torsion classes.

  \begin{figure}[htp]
    \begin{tikzpicture}[scale = .6]
      \node (0) at (0,0) {0};
      \node[rectangle,draw] (1) at (4,2) {$\sst{1}$};
      \node[rectangle,draw] (2) at (-4,2) {$\sst{2}$};
      \node[rectangle,draw] (3) at (0,2) {$\sst{3}$};
      \node[block=2] (323) at (-3,4) {\nodepart{one}$\sst{3 \\ 2}$\nodepart{two}$\sst{3}$};
      \node[block=2] (221) at (0,5) {\nodepart{one}$\sst{2}$\nodepart{two}$\sst{2 \\ 1}$};
      \node[block=2] (13) at (2,4) {\nodepart{one}$\sst{1}$\nodepart{two}$\sst{3}$};
      \node[block=2] (232) at (-6,6) {\nodepart{one}$\sst{2}$\nodepart{two}$\sst{3 \\ 2}$};
      \node[block=3] (22132) at (-4,8) {\nodepart{one}$\sst{2}$\nodepart{two}$\sst{2 \\ 1}$\nodepart{three}$\sst{3 \\ 2}$};
      \node[block=3] (1323) at (0,8) {\nodepart{one}$\sst{1}$\nodepart{two}$\sst{3 \\ 2}$\nodepart{three}$\sst{3}$};
      \node[block=2] (121) at (4,8) {\nodepart{one}$\sst{1}$\nodepart{two}$\sst{2 \\ 1}$};
      \node[block=3] (12132) at (0,10) {\nodepart{one}$\sst{1}$\nodepart{two}$\sst{2 \\ 1}$\nodepart{three}$\sst{3 \\ 2}$};
      \node[block=2,red] (132) at (5,5) {\nodepart{one}$\sst{1}$\nodepart{two}$\sst{3 \\ 2}$};
      \node[rectangle,draw,red] (32) at (2,2) {$\sst{3 \\ 2}$};
      \node[block=2,red] (213) at (-6,4) {\nodepart{one}$\sst{2 \\ 1}$\nodepart{two}$\sst{3}$};
      \node[rectangle,draw,red] (21) at (-2,2) {$\sst{2 \\ 1}$};

      \draw[->] (12132) -- (22132);
      \draw[->] (12132) -- (1323);
      \draw[->] (12132) -- (121);
      \draw[->] (22132) -- (232);
      \draw[->] (22132) -- (221);
      \draw[->] (1323) -- (323);
      \draw[->] (1323) -- (13);
      \draw[->] (121) -- (221);
      \draw[->] (121) -- (1);
      \draw[->] (232) -- (323);
      \draw[->] (232) -- (2);
      \draw[->] (221) -- (2);
      \draw[->] (323) -- (3);
      \draw[->] (13) -- (1);
      \draw[->] (13) -- (3);
      \draw[->] (2) -- (0);
      \draw[->] (3) -- (0);
      \draw[->] (1) -- (0);
      \draw[->] (1323) -- (132);
      \draw[->] (132) -- (1);
      \draw[->] (132) -- (32);
      \draw[->] (323) -- (32);
      \draw[->] (32) -- (0);
      \draw[->] (22132) -- (213);
      \draw[->] (213) -- (21);
      \draw[->] (213) -- (3);
      \draw[->] (221) -- (21);
      \draw[->] (21) -- (0);
    \end{tikzpicture}
    \caption{$\Hasse(\ice\Lambda)$, where the vertices are the corresponding wide $\tau$-tilting modules}
    \label{fig:nakice}
  \end{figure}
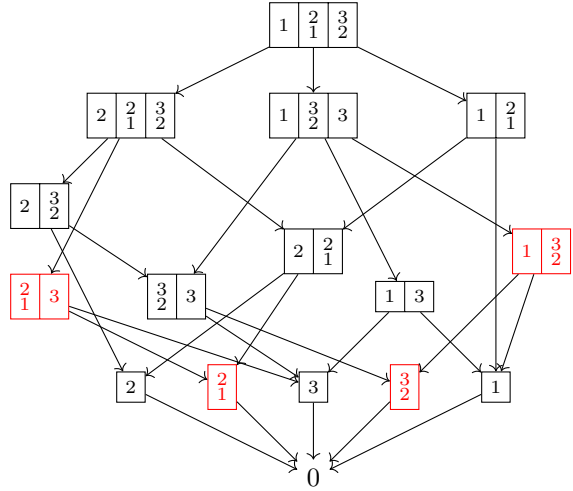

  We can check that $\Hasse(\ice\Lambda)$ actually contains $\Hasse(\tors\Lambda)$ as a full subquiver as claimed in Proposition \ref{prop:asfullsub}.
  Besides, observe that $\Hasse(\ice\Lambda)$ satisfies the same property as Corollary \ref{cor:arrow}, that is, there are exactly $|M|$ arrows which start at $M$ for every $M \in \wttilt\Lambda$.
\end{example}
The last observation in the previous example raises the following question.
\begin{question}\label{question}
  Which classes of algebras $\Lambda$ satisfy the following property: there are exactly $|M|$ arrows in $\Hasse(\wttilt\Lambda)$ which start at $M$ for every $M \in \wttilt\Lambda$?
\end{question}
Corollary \ref{cor:arrow} shows that every hereditary algebra satisfies this property, and the previous example also satisfies it.
If there is a theory of \emph{mutations of wide $\tau$-tilting $\Lambda$ modules} as in the hereditary case exists, then $\Lambda$ satisfies this property.
Unfortunately, the previous example shows that one cannot expect such a theory similar to the hereditary case. Indeed, there is an arrow $\sst{2} \oplus \sst{2\\1} \oplus \sst{3\\2} \to \sst{2\\1} \oplus \sst{3}$ in Figure \ref{fig:nakice}.
Along this arrow, two indecomposable direct summands change, namely, $\sst{2}$ and $\sst{3\\2}$. Thus it seems that one cannot interpret this arrow as representing a mutation of $\sst{2} \oplus \sst{2\\1} \oplus \sst{3\\2}$ at one indecomposable direct summand.

\begin{example}\label{ex:nonnak}
  Let $Q$ be the following quiver
  \[
  \begin{tikzcd}[row sep=0]
    1 \rar & 2 \rar & 3 \rar[loop right]
  \end{tikzcd}
  \]
  and consider $\Lambda':=kQ/I$, where $I$ is a two-sided ideal generated by all the paths of length 2.
  Then the Auslander-Reiten quiver of $\mod\Lambda'$ is as follows:
  
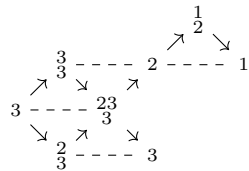
\[
  \begin{tikzpicture}[scale=.6]
   \node (3) at (-1,-1) {$\sst{3}$};
   \node (33) at (0,0) {$\sst{3 \\ 3}$};
   \node (23) at (0,-2) {$\sst{2 \\ 3}$};
   \node (233) at (1,-1) {$\sst{2 3 \\ 3}$};
   \node (2) at (2,0) {$\sst{2}$};
   \node (3') at (2,-2) {$\sst{3}$};
   \node (12) at (3,1) {$\sst{1 \\ 2}$};
   \node (1) at (4,0) {$\sst{1}$};

   \draw[->] (3) -- (33);
   \draw[->] (33) -- (233);
   \draw[->] (3) -- (23);
   \draw[->] (23) -- (233);
   \draw[->] (233) -- (2);
   \draw[->] (233) -- (3');
   \draw[->] (2) -- (12);
   \draw[->] (12) -- (1);

   \draw[dashed] (3) -- (233);
   \draw[dashed] (33) -- (2);
   \draw[dashed] (23) -- (3');
   \draw[dashed] (2) -- (1);
  \end{tikzpicture}
  \]
  We depict $\Hasse(\tors\Lambda')$ with the brick labeling in Figure \ref{fig:hassesttilt}.
  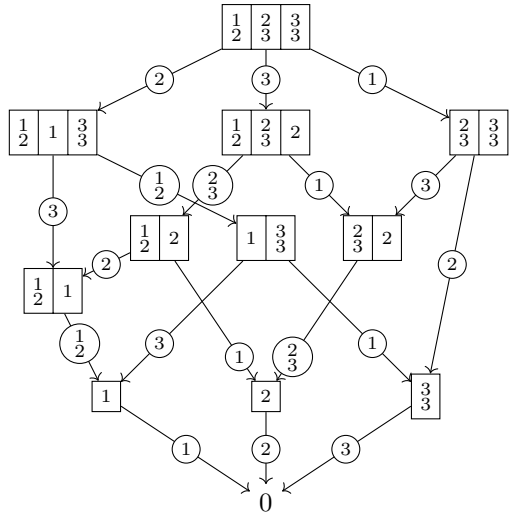
\begin{figure}[htp]
    \begin{tikzpicture}[scale=.7]
      \node (0) at (3,-1) {0};
      \node[rectangle,draw] (33) at (6,1) {$\sst{3 \\ 3}$};
      \node[rectangle,draw] (1) at (0,1) {$\sst{1}$};
      \node[rectangle,draw] (2) at (3,1) {$\sst{2}$};
      \node[block=2] (133) at (3,4) {\nodepart{one}$\sst{1}$\nodepart{two}$\sst{3 \\ 3}$};
      \node[block=2] (122) at (1,4) {\nodepart{one}$\sst{1 \\ 2}$\nodepart{two}$\sst{2}$};
      \node[block=2] (232) at (5,4) {\nodepart{one}$\sst{2 \\ 3}$\nodepart{two}$\sst{2}$};
      \node[block=2] (121) at (-1,3) {\nodepart{one}$\sst{1 \\ 2}$\nodepart{two}$\sst{1}$};
      \node[block=3] (12133) at (-1,6) {\nodepart{one}$\sst{1 \\ 2}$\nodepart{two}$\sst{1}$\nodepart{three}$\sst{3 \\ 3}$};
      \node[block=3] (12232) at (3,6) {\nodepart{one}$\sst{1 \\ 2}$\nodepart{two}$\sst{2 \\ 3}$\nodepart{three}$\sst{2}$};
      \node[block=2] (2333) at (7,6) {\nodepart{one}$\sst{2 \\ 3}$\nodepart{two}$\sst{3 \\ 3}$};
      \node[block=3] (122333) at (3,8) {\nodepart{one}$\sst{1 \\ 2}$\nodepart{two}$\sst{2 \\ 3}$\nodepart{three}$\sst{3 \\ 3}$};

      \node[blabel] (33b0) at (4.5,0) {$\sst{3}$};
      \node[blabel] (122333b12133) at (1,7) {$\sst{2}$};
      \node[blabel] (2b0) at (3,0) {$\sst{2}$};
      \node[blabel] (122333b2333) at (5,7) {$\sst{1}$};
      \node[blabel] (1b0) at (1.5,0) {$\sst{1}$};
      \node[blabel] (122333b12232) at (3,7) {$\sst{3}$};
      \node[blabel] (12133b121) at (-1,4.5) {$\sst{3}$};
      \node[blabel] (12133b133) at (1,5) {$\sst{1 \\ 2}$};
      \node[blabel] (12232b122) at (2,5) {$\sst{2 \\ 3}$};
      \node[blabel] (12232b232) at (4,5) {$\sst{1}$};
      \node[blabel] (2333b232) at (6,5) {$\sst{3}$};
      \node[blabel] (2333b33) at (6.5,3.5) {$\sst{2}$};
      \node[blabel] (122b121) at (0,3.5) {$\sst{2}$};
      \node[blabel] (122b2) at (2.5,1.75) {$\sst{1}$};
      \node[blabel] (232b2) at (3.5,1.75) {$\sst{2 \\ 3}$};
      \node[blabel] (121b1) at (-0.5,2) {$\sst{1 \\ 2}$};
      \node[blabel] (133b1) at (1,2) {$\sst{3}$};
      \node[blabel] (133b33) at (5,2) {$\sst{1}$};

      \draw[-] (122333) -- (122333b12133);
      \draw[-] (122333) -- (122333b12232);
      \draw[-] (122333) -- (122333b2333);
      \draw[-] (12133) -- (12133b121);
      \draw[-] (12133) -- (12133b133);
      \draw[-] (12232) -- (12232b122);
      \draw[-] (12232) -- (12232b232);
      \draw[-] (2333) -- (2333b232);
      \draw[-] (2333) -- (2333b33);
      \draw[-] (122) -- (122b121);
      \draw[-] (122) -- (122b2);
      \draw[-] (133) -- (133b1);
      \draw[-] (133) -- (133b33);
      \draw[-] (232) -- (232b2);
      \draw[-] (121) -- (121b1);
      \draw[-] (1) -- (1b0);
      \draw[-] (2) -- (2b0);
      \draw[-] (33) -- (33b0);

      \draw[->] (122333b12133) -- (12133);
      \draw[->] (122333b12232) -- (12232);
      \draw[->] (122333b2333) -- (2333);
      \draw[->] (12133b121) -- (121);
      \draw[->] (12133b133) -- (133);
      \draw[->] (12232b122) -- (122);
      \draw[->] (12232b232) -- (232);
      \draw[->] (2333b232) -- (232);
      \draw[->] (2333b33) -- (33);
      \draw[->] (122b121) -- (121);
      \draw[->] (122b2) -- (2);
      \draw[->] (133b1) -- (1);
      \draw[->] (133b33) -- (33);
      \draw[->] (232b2) -- (2);
      \draw[->] (121b1) -- (1);
      \draw[->] (1b0) -- (0);
      \draw[->] (2b0) -- (0);
      \draw[->] (33b0) -- (0);
    \end{tikzpicture}
    \caption{$\Hasse(\tors\Lambda')$}
    \label{fig:hassesttilt}
  \end{figure}

  We can compute all ICE intervals, wide $\tau$-tilting modules and ICE-closed subcategories in the same way as in Example \ref{ex:nak}, which we omit. Then we obtain $\Hasse(\ice\Lambda')$ as we depict in Figure \ref{fig:hasseice}, where we represent each ICE-closed subcategory by the corresponding wide $\tau$-tilting module, and the red vertices are those which are not support $\tau$-tilting modules.
  We can see from this Hasse quiver that this $\Lambda'$ also satisfies the property in Question \ref{question}.
  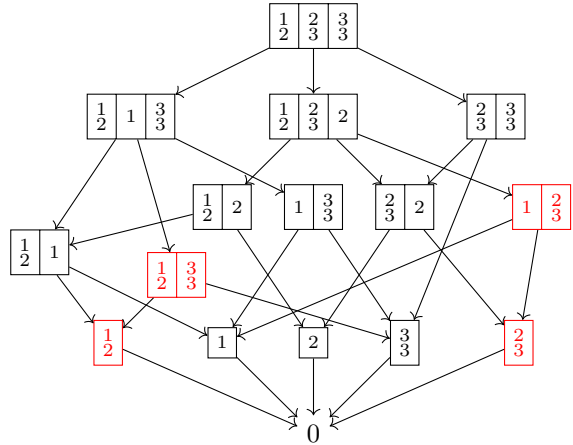
\begin{figure}[htp]
    \begin{tikzpicture}[scale=.6]
      \node (0) at (3,-1) {0};
      \node[rectangle,draw] (3) at (5,1) {$\sst{3 \\ 3}$};
      \node[rectangle,draw] (1) at (1,1) {$\sst{1}$};
      \node[rectangle,draw] (2) at (3,1) {$\sst{2}$};
      \node[rectangle,draw,red] (12) at (-1.5,1) {$\sst{1 \\ 2}$};
      \node[rectangle,draw,red] (23) at (7.5,1) {$\sst{2 \\ 3}$};
      \node[block=2] (1a3) at (3,4) {\nodepart{one}$\sst{1}$\nodepart{two}$\sst{3 \\ 3}$};
      \node[block=2,red] (12a3) at (0,2.5) {\nodepart{one}$\sst{1 \\ 2}$\nodepart{two}$\sst{3 \\ 3}$};
      \node[block=2] (1a2) at (1,4) {\nodepart{one}$\sst{1 \\ 2}$\nodepart{two}$\sst{2}$};
      \node[block=2] (2a23) at (5,4) {\nodepart{one}$\sst{2 \\ 3}$\nodepart{two}$\sst{2}$};
      \node[block=2,red] (1a23) at (8,4) {\nodepart{one}$\sst{1}$\nodepart{two}$\sst{2 \\ 3}$};
      \node[block=2] (1a12) at (-3,3) {\nodepart{one}$\sst{1 \\ 2}$\nodepart{two}$\sst{1}$};
      \node[block=3] (1a12a3) at (-1,6) {\nodepart{one}$\sst{1 \\ 2}$\nodepart{two}$\sst{1}$\nodepart{three}$\sst{3 \\ 3}$};
      \node[block=3] (1a2a23) at (3,6) {\nodepart{one}$\sst{1 \\ 2}$\nodepart{two}$\sst{2 \\ 3}$\nodepart{three}$\sst{2}$};
      \node[block=2] (2a3) at (7,6) {\nodepart{one}$\sst{2 \\ 3}$\nodepart{two}$\sst{3 \\ 3}$};
      \node[block=3] (1a2a3) at (3,8) {\nodepart{one}$\sst{1 \\ 2}$\nodepart{two}$\sst{2 \\ 3}$\nodepart{three}$\sst{3 \\ 3}$};

      \draw[->] (1a2a3) -- (1a12a3);
      \draw[->] (1a2a3) -- (1a2a23);
      \draw[->] (1a2a3) -- (2a3);
      \draw[->] (1a12a3) -- (1a12);
      \draw[->] (1a12a3) -- (1a3);
      \draw[->] (1a2a23) -- (1a2);
      \draw[->] (1a2a23) -- (2a23);
      \draw[->] (2a3) -- (2a23);
      \draw[->] (2a3) -- (3);
      \draw[->] (1a2) -- (1a12);
      \draw[->] (1a2) -- (2);
      \draw[->] (1a3) -- (1);
      \draw[->] (1a3) -- (3);
      \draw[->] (2a23) -- (2);
      \draw[->] (1a12) -- (1);
      \draw[->] (1) -- (0);
      \draw[->] (2) -- (0);
      \draw[->] (3) -- (0);
      \draw[->] (1a12a3) -- (12a3);
      \draw[->] (12a3) -- (12);
      \draw[->] (1a12) -- (12);
      \draw[->] (12) -- (0);
      \draw[->] (12a3) -- (3);
      \draw[->] (1a2a23) -- (1a23);
      \draw[->] (1a23) -- (1);
      \draw[->] (1a23) -- (23);
      \draw[->] (2a23) -- (23);
      \draw[->] (23) -- (0);
    \end{tikzpicture}
    \caption{$\Hasse(\ice\Lambda')$}
    \label{fig:hasseice}
  \end{figure}
\end{example}
Due to the lack of the theory of mutations of wide $\tau$-tilting modules, we computed the Hasse quivers of ICE-closed subcategories directly by computing all indecomposables in each ICE-closed subcategory. This raises the following question.
\begin{question}
  Is there any theoretic interpretation of arrows in $\Hasse(\ice\Lambda)$?
\end{question}

\begin{ack}
  The authors would like to thank their supervisor Hiroyuki Nakaoka for his encouragement and useful comments.
  They would also like to thank Osamu Iyama for helpful discussions.
  The first author is supported by JSPS KAKENHI Grant Number JP18J21556.
\end{ack}

\end{document}